\documentclass{birkjour}
\usepackage[colorlinks=true]{hyperref} %,pagebackref
\usepackage{amssymb}
\usepackage{amsthm}
\usepackage{amsmath}
\usepackage[active]{srcltx}%(Para busqueda inversa)

\usepackage{psfig}
\usepackage{hyperref}

\usepackage{xcolor,enumerate,array}
\usepackage[color,all]{xy}

\input xy
\xyoption{all}

\newtheorem{thm}{Theorem}[section]
\newtheorem{cor}[thm]{Corollary}
\newtheorem{lem}[thm]{Lemma}
\newtheorem{prop}[thm]{Proposition}
\theoremstyle{definition}
\newtheorem{defn}[thm]{Definition}
\theoremstyle{remark}
\newtheorem{rem}[thm]{Remark}
\newtheorem*{ex}{Example}
\numberwithin{equation}{section}

\numberwithin{equation}{section}
\numberwithin{thm}{section}

\numberwithin{equation}{section}
\numberwithin{thm}{section}
\begin{document}

%\begin{frontmatter}



\title[]
{On the geometry of the Clairin theory of conditional symmetries for higher-order systems of PDEs with applications}

%----------Author 1
\author[A.M. Grundland]{A.M. Grundland}
\address{Centre de Recherches Math\'ematiques, Universit\'e de Montr\'eal,\\
	Succ. Centre-Ville, CP 6128, Montr\'eal (QC) H3C 3J7, Canada \\
	Department of Mathematics and Computer Science, Universit\'e du Qu\'ebec \`a Trois-Rivi\`eres, \\
	CP 500, Trois-Rivi\`eres (QC) G9A 5H7, Canada}

\email{grundlan@crm.umontreal.ca}

%\thanks{This work was completed with the support of our
%	\TeX-pert.}
%----------Author 2

\author{J. de Lucas}
\address{Department of Mathematical Methods in Physics, University of Warsaw, \\ ul. Pasteura 5, 02-093, Warszawa, Poland.}
\email{javier.de.lucas@fuw.edu.pl}
\subjclass{{35Q53} (primary); 35Q58, 53A05 (secondary)}

\keywords{Lie point symmetries, conditional symmetries, contact forms, PDE Lie systems, jet bundles, Clairin formalism, nonlinear wave equation, Gauss--Codazzi equations, generalised Liouville equation}

\date{}

%\title{On the geometry of the Clairin theory of conditional symmetries\\ for higher-order systems of PDEs with applications}


%\author{A.M. Grundland}

\begin{abstract}This work presents a geometrical formulation of the Clairin theory of conditional symmetries for higher-order systems of partial differential equations (PDEs). We devise methods for obtaining Lie algebras of conditional symmetries from known conditional symmetries, and unnecessary previous assumptions of the theory are removed. As a consequence, new insights into other types of conditional symmetries arise.  We then apply the so-called PDE Lie systems to the derivation and analysis of Lie algebras of conditional symmetries. In particular, we develop a method for obtaining solutions of a higher-order system of PDEs via the solutions and geometric properties of a PDE Lie system, whose form gives a Lie algebra of conditional symmetries of the Clairin type. Our methods are illustrated with physically relevant examples such as nonlinear wave equations, the Gauss--Codazzi equations for minimal soliton surfaces, and generalised Liouville equations. %Superposition rules or $t$-dependent superposition rules for PDE Lie systems are employed to obtain B\"acklund transformations.  %Relevantly, we show for the first time in the literature that Lie systems of partial differential equations admit solutions of kink type and we use them to obtain conditional symmetries of the Colapso equation. 
\end{abstract}
\maketitle

%\begin{keyword} Lie point symmetries \sep conditional symmetries \sep contact forms \sep PDE Lie systems \sep jet bundles \sep Clairin formalism \sep nonlinear wave equation \sep Gauss--Codazzi equations \sep generalised Liouville equation
%
%%% keywords here, in the form: keyword \sep keyword
%        
%\MSC {35Q53} (primary) \sep 35Q58 \sep 53A05 (secondary)
%%% MSC codes here, in the form: \MSC code \sep code
%%% or \MSC[2008] code \sep code (2000 is the default)
%
%\end{keyword}

%\end{frontmatter}

\section{Introduction}\label{Introduction}\setcounter{equation}{0}
Over the last two centuries, Lie's theory of symmetries of partial differential equations (PDEs) \cite{Li80} has been the subject of extensive research in  mathematics and physics \cite{Li80,Ol93,St89,PW83}. During those years, the development of this theory has led to significant progress in classifying and solving differential equations, yielding many new interesting results (see e.g. \cite{ITV93} and references therein). A number of attempts to generalise this subject and to develop its applications can be found in the literature \cite{BC69,CG01,CGM04,CM04,Cl03,Fu93,GM04,GS07,GMR97,ITV93,MR03,MR08,Ol92,OR86,VK99,PW83,Zi93}. 

Of particular interest from a physical point of view has been the development of the theory of conditional symmetries which evolved in the process of extending Lie's classical theory of symmetries of PDEs \cite{BC69,Cl03,GMR97,GR95}. This approach consists essentially of supplementing the original system of PDEs with certain first-order differential constraints (DCs) for which a symmetry criterion is applied. As a result we obtain an overdetermined system of PDEs admitting, in some cases, a larger class of Lie point symmetries than the original system. Consequently, this approach enables us to construct new classes of solutions of the original system. The problem of the determination of such classes of explicit solutions was first solved by Clairin \cite{Cl03} in 1903 by subjecting the original system of PDEs to several DCs.
Since then, generalisations of conditional symmetries have been formulated by many authors \cite{BC69,GMR97,GR95,Ol92,OR86,OR87,Zi93}.

In the {\it Clairin theory of conditional symmetries} \cite{Cl03,GMR97,GR95}, 
the DCs are given by an {\it integrable} (in the sense of fulfilling a {\it compatibility condition} explained later and in \cite{Du10,Gu08}) first-order system of PDEs in {\it normal form}, i.e. such that the derivatives of particular solutions are functions of the dependent and independent variables, and  compatible  with the initial system of PDEs. Although there exist more general conditional symmetry methods based on adding DCs that need not give rise to first-order systems of PDEs in normal form \cite{CK89,Fu93,LW89,Ol92,OR87,PS92,Zi93}, the Clairin theory allows us to study the solutions of the DCs through many techniques. For instance, one can use different types of {\it Lie systems} \cite{GMR97,GR95} or a naturally related Abelian Lie algebra of Lie point symmetries, which generates new particular solutions of the initial system of PDEs from known ones satisfying the given DCs. 

In the Clairin theory, the DCs are determined by the zeroes of the characteristics of a Lie algebra $L$ of vector fields \cite{Ol93}. The Lie algebra $L$ consists of Lie symmetries, referred to as {\it Clairin conditional symmetries}, of the overdetermined system formed by the original system subjected to such DCs. The crux of the Clairin theory is to find $L$ \cite{GMR97,GR95,Ol92}. To simplify the task, $L$ is assumed to be Abelian and to admit a basis of a particular type \cite{GMR97,GR95}. Then, $L$ can be obtained by solving a nonlinear system of PDEs \cite{GR95,Ol92}.

The Clairin method, which is mainly described in terms of coordinates (cf. \cite{Cl03,GMR97,GR95}), lacks an intrinsic geometric formulation.  Additionally, there exists to our knowledge no detailed geometric formalism for studying higher-order systems of PDEs in this context, as works have, until now, focused on the formalism for first-order systems of PDEs \cite{Cl03,GR95} or on applying conditional symmetry techniques to  higher-order systems of PDEs without a detailed theoretical analysis \cite{Cl03,GMR97,GR95}.

Hence, our first aim is  to present a geometric Clairin theory of conditional symmetries for higher-order systems of PDEs. This allows us to avoid most previous  assumptions on $L$, to clarify some of the geometric features \cite{GMR97,GR95}, and, as a consequence, to provide new insights into other conditional symmetries and related structures \cite{Ol92,Zi93}. Our theory also provides methods for constructing new Lie algebras of conditional symmetries from known ones. Our second aim is to develop methods for obtaining conditional symmetries through the so-called {\it PDE Lie systems} and to apply them to the study of physical systems \cite{Dissertationes}.

Geometrically, the Clairin theory of conditional symmetries  for higher-order systems of PDEs can be summarised as follows. An $n$-th order system of PDEs whose dependent and independent variables are functions on $U$ and $X$, respectively, amounts to a subset $\mathcal{S}_\Delta$ of the $n$-th order jet bundle, $J^n$, of the bundle $\pi^0_{-1}:(x,u)\in X\times U\mapsto x\in X$ \cite{Ol93}. Lie algebras of conditional symmetries of $\mathcal{S}_\Delta$ are given by a Lie algebra $L$ of vector fields on $J^0=X\times U$ defining a submanifold $\mathcal{S}^n_L\subset J^n$, the so-called {\it characteristic system} of $L$, given by the common zeroes of the {\it total differentials} \cite{Ol93} of the characteristics of the elements of $L$ in such a way that $L$ consists of Lie point symmetries of the system of PDEs related to $\mathcal{S}^n_L\cap \mathcal{S}_\Delta$. The Clairin theory requires that $\mathcal{S}^n_L$ be related to a first-order  system of PDEs in normal form satisfying integrability conditions (sometimes called the {\it compatibility conditions}) \cite{Du10,Gu08}. If $L$ is Abelian and admits a basis of a particular type \cite{GMR97,GR95},  then $L$ can be derived by solving a nonlinear system of PDEs \cite{GMR97,GR95,Ol92}.

Apart from providing a careful geometric Clairin theory for conditional symmetries of higher-order systems of PDEs, let us describe other new contributions of our work.
 
First, by describing the characteristics of vector fields via the {\it contact forms} on $J^n$ \cite{Ol93,St89}, the characteristic system $\mathcal{S}^n_L$ for a linear space $L$ of vector fields on $J^0$ is given in Definition \ref{CharSys} or, equivalently, in Definition \ref{CharSysII} in an intrinsic geometrical way. This geometrises and generalises several types of Lie algebras of conditional vector fields \cite{GMR97,GR95,Zi93}. 

Theorems \ref{Th:SecSL1} and \ref{Th:SecSL} characterise when $\mathcal{S}^n_L$ can be described as a system of PDEs in normal form. If $\mathcal{S}^n_L$ is assumed to be a system of PDEs in normal form, then Theorem \ref{Th:ProlSL} provides necessary and sufficient conditions to ensure that the prolongations to $J^n$ (see \cite{Ol93}) of the vector fields of $L$ are tangent to $\mathcal{S}^n_L$. This fulfils results given in the previous literature,  where only necessary conditions are detailed \cite{GMR97,GR95,Ol92}. Remarkably, $L$ need not be a Lie algebra, as assumed previously \cite{GMR97,GR95}. Since our results state that $L$ must span a distribution $\mathcal{D}^L$ of dimension $q$ projecting onto $TX$, one obtains that this condition, in coordinates, amounts to the existence of a non-degenerate $p\times p$ matrix $\xi$ appearing in the Clairin formalism \cite{GMR97,GR95}.

Subsequently, Proposition \ref{Prop:SLInt} details the necessary and sufficient conditions on $L$ to ensure that $\mathcal{S}^n_L$, related to the system in normal form, is  {\it locally solvable} \cite{Ol93}. Corollary \ref{NorRel} shows that the standard conditions found in the literature \cite{Ol92} ensure that $\mathcal{S}^n_L$ is locally solvable. Remarkably, if $n>1$, then Theorem \ref{Th:SecSL} and Proposition \ref{Prop:SLInt} ensure that if $\mathcal{S}^n_L$ is a system of PDEs in normal form, then $\mathcal{S}^n_L$ is locally solvable, which also implies that the projection of $\mathcal{S}^n_L$ to $J^1$ satisfies the compatibility conditions \cite{Du10,Gu08}. This implies that the solutions of $\mathcal{S}_{\Delta}\cap \mathcal{S}^n_L$ are the same as those of the system of PDEs given by $ \mathcal{S}_{\Delta}\cap\mathcal{S}^1_L$ (where $\mathcal{S}^1_L$ is considered as a subspace of $J^n$ in the natural way \cite{VK99}). This is the usual approach appearing in the literature (see \cite{GMR97,GR95} and Section \ref{Appl}).

The Clairin approach to conditional symmetries  assumes that $L$ has a basis of a particular form \cite{Cl03,GMR97,GR95}.  Section \ref{Sec:NPDEL} characterises geometrically when $L$ admits such a basis. Such Lie algebras are called {\it rectified PDE Lie algebras}. Next, it is studied when a linear space of vector fields $L'$ is such that $\mathcal{S}^n_{L'}=\mathcal{S}^n_L$. If $L'$ is a Lie algebra, then it is called a {\it rectifiable PDE Lie algebra}.  

Next, we survey the  geometric properties of conditional Lie symmetries and raise certain technical questions frequently overlooked in the literature. We show that every normal PDE Lie 
algebra of conditional symmetries $L$ is such that every rectifiable PDE Lie algebra of the vector fields $L'$ satisfying $\mathcal{S}^n_{L'}=\mathcal{S}^n_L$  is also a Lie algebra of conditional Lie symmetries of the same system of PDEs. This allows us to obtain new Lie algebras of conditional symmetries from known ones.  

Finally, we study the differential equations characterising rectified PDE Lie algebras of conditional symmetries. To solve them, we extend the methods of \cite{GMR97,GR95} and provide conditions to ensure that rectified PDE Lie algebras can be derived via the so-called {\it PDE Lie systems} \cite{CGM07,OG00}. PDE Lie systems are first-order systems of PDEs in normal form admitting a {\it superposition rule}, i.e. a function allowing us to obtain their general solutions in terms of a generic family of particular solutions and some constants \cite{CGM07,OG00}. These systems of PDEs have attracted much attention lately  and their special structure allows for the development of methods to study their solutions \cite{CGL18,Dissertationes}.

Our use of PDE Lie systems to obtain Lie algebras of conditional symmetries is more general than the approach used in \cite{GMR97,GR95}, where only particular types of PDE Lie systems or standard Lie systems appear. Moreover, we provide assumptions on higher-order systems of PDEs, $\mathcal{S}_\Delta$, that enable us to construct a PDE Lie system describing some of its conditional symmetries (in the Clairin context) and whose solutions are solutions of $\mathcal{S}_\Delta$.

As applications, PDE Lie systems  are employed to study nonlinear wave equations \cite{GR95} and minimal surfaces for {\it Gauss--Codazzi equations} \cite{Sc41}. The fact that PDE Lie systems related to solvable Vessiot--Guldberg Lie algebras can be solved is employed to obtain minimal surfaces of Gauss--Codazzi equations (cf. \cite{CGM07,Dissertationes,OG00}). Finally, some solutions of the generalised Liouville equations \cite{Sa96} are provided. 

This paper is organised as follows. Section 2 describes the basic geometric tools employed in the work. Section 3 is concerned with the definition of a {\it characteristic system} for a linear space of vector fields on $X\times U$. Section 4 establishes when $\mathcal{S}^n_L$ amounts to an $n$-th order system of PDEs in normal form. Section 5 analyses certain geometric properties of characteristic systems. Section 6 focuses on the generality of a certain type of Lie algebra appearing in the Clairin theory of conditional symmetries. Section 7 studies Lie algebras of Clairin conditional symmetries. The use of PDE Lie systems to obtain conditional symmetries is developed in Section 8. In Section 9 several applications of our results are presented. Finally,  Section 10 summarises our new results, and  acknowledgements are detailed in Section 11.

\section{Geometric preliminaries on jet bundles and systems of PDEs}\label{GeomPre}\setcounter{equation}{0}

This section presents the notions of the theory of jet bundles and the notation to be employed in this work. This allows us to make our work more self-contained and easier to follow. Unless otherwise stated, we assume that mathematical structures are smooth and globally defined. 

Let $X$ and $U$ be manifolds of dimension $p$ and $q$, respectively.  We define $J^0=X\times U$ and $J^{-1}=X$, where $J^n$ stands for the $n$-th order jet bundle of the  trivial bundle  given by the projection map $\pi^0_{-1}:
X\times U\rightarrow X$ onto $X$. Let $\{x^1,\ldots, x^p\}$ and $\{u^1,\ldots,u^q\}$ be coordinate systems on $X$ and $U$, respectively. Such coordinate systems induce coordinates $u^\alpha_K$ on $J^n$, where $K=(k_1,\ldots,k_p)$ is a multi-index with $k_1,\ldots,k_p\in\mathbb{N}\cup \{0\}$  and $ |K|=\sum_{i=1}^pk_i\leq n$. Then, $\{x^i,u^\alpha,u^\alpha_K\}$, with $1\leq |K|\leq n$, becomes a local coordinate system of $J^n$. We write $u^\alpha_{K,i}$, with $i=1,\ldots,p$, for the coordinate $u^\alpha_{K'}$ where $k'_i=k_i+1$ and $k_j=k'_j$ for the indices $j\neq i$. 

If $n,m$ are integers and $n\geq m\geq -1$, then $J^n$ is a bundle over $J^m$ relative to the natural projection $\pi^n_m:J^n\rightarrow J^m$ and we write $\Gamma(\pi^n_m)$ for its space of sections.

We denote by ${\bf j}^n_x\sigma=(x,u,u^{(n)})$ an arbitrary point of $J^n$. Given a section $\sigma\in \Gamma(\pi^0_{-1})$, say $\sigma(x)=(x,u(x))$, its {\it prolongation} to $J^n$  is the section ${\bf j}^n\sigma\in \Gamma(\pi^n_{-1})$ of the form ${\bf j}^n\sigma(x)=(x^i,u^\alpha (x), \frac{\partial^r u^\alpha}{\partial x^{k_1}\ldots\partial x^{k_p}}(x))$ for $1\leq r=|(k_1,\ldots,k_p)|\leq n$.  Sections  $\sigma^{(n)}\in\Gamma(\pi^n_{-1})$ of the form $\sigma^{(n)}={\bf j}^n\sigma$ for a $\sigma\in \Gamma(\pi^0_{-1})$ are called {\it holonomic}. 

The main geometric structure on $J^n$, the so-called {\it Cartan distribution} $\mathcal{C}^{(n)}$, is the smallest distribution on $J^n$ tangent to all prolongations ${\bf j}^n\sigma$ for an arbitrary $\sigma\in\Gamma(\pi^0_{-1})$ \cite{Ol93}. In coordinates, $\mathcal{C}^{(n)}$ is spanned by the vector fields
$$
D_i=\frac{\partial}{\partial x^i}+\sum_{|K|\leq n-1}\sum_{\alpha=1}^qu^\alpha_{K,i}\frac{\partial}{\partial u^\alpha_K},\quad V_{\alpha K'}=\frac{\partial}{\partial u^\alpha_{K'}},\quad $$
where $i=1,\ldots,p$, $\alpha=1,\ldots,q$, $|K'|=n$, and $D_i$ is called the {\it total derivative} relative to  $x^i$ on $J^n$. The Cartan distributions are not involutive \cite{Ol93}. 
We write $D_K=D^{k_1}_{1}\circ\cdots\circ D^{k_p}_{p}$ for a multi-index $K$ and $D_K={\rm Id}$ when $|K|=0$.
 
A {\it contact form} on $J^n$ is a one-form $\theta$ on $J^n$ such that ${\bf j}^{n}\sigma^*\theta=0$ for every $\sigma\in\Gamma(\pi^0_{-1})$. Hence, contact forms allow us to determine when a section $\sigma^{(n)}\in \Gamma(\pi^n_{-1})$ is holonomic.  We write $\mathcal{CF}(J^n)$ for the space of contact forms on $J^n$. In particular, the one-forms
\begin{equation}\label{def:BasCon}
\theta^\alpha_K=du^\alpha_K-\sum_{i=1}^pu^\alpha_{K,i} dx^i,\qquad \alpha=1,\ldots, q,\qquad 0\leq |K|\leq n-1,
\end{equation}
are called the {\it basic contact forms} on $J^n$ relative to the coordinate system $\{x^i,u^\alpha,u^\alpha_K\}$, with $1\leq|K|\leq n$. On an open set of $J^n$ with coordinates $\{x^i,u^\alpha,u^\alpha_K\}$, contact forms are linear combinations with functions on $J^n$ of (\ref{def:BasCon}). In this sense, the basic contact forms (\ref{def:BasCon}) form a basis of $\mathcal{CF}(J^n)$.

A vector field $Y$ on $J^0$ can be written in local coordinates as
\begin{equation}\label{Y}
        Y=\sum_{i=1}^p\xi^i\frac{\partial}{\partial x^i}+\sum_{\alpha=1}^q\varphi^\alpha\frac{\partial}{\partial u^\alpha},
\end{equation}
where  $\xi^i,\varphi^\alpha$ are univocally defined functions on $J^0$. The {\it prolongation of $Y$ to $J^n$} is the only vector field, ${\bf j}^nY$, on $J^n$ leaving the space of vector fields taking values in $\mathcal{C}^{(n)}$ invariant (relative to the Lie bracket of vector fields) and projecting onto $Y$ via $\pi^n_0:J^n\rightarrow J^{0}$. Consequently, if  $\theta$ is a contact form on $J^n$, then the Lie derivative $\mathcal{L}_{{\bf j}^nY}\theta$ is also a contact form. The expression of $J^nY$ in coordinates\footnote{To provide an expression easy to deal with, it is assumed that functions and total derivatives are defined on $J^{n+1}$. Despite that, ${\bf j}^nY$  is a well-defined vector field on $J^n$. This approach is standard in the literature (cf. \cite[p. 110]{Ol93}).} reads \cite{Ol93,St89}
\begin{equation}\label{Prol}
\begin{gathered}
        {\bf j}^n\,Y=Y+\!\!\!\sum_{1\leq |K|\leq n}\sum_{\alpha=1}^q\psi_{YK}^\alpha\frac{\partial}{\partial u_K^\alpha},\quad 
        \psi^\alpha_{YK}=D_K Q_Y^\alpha+\sum_{i=1}^p\xi^i u^\alpha_{K,i}, \\ Q_Y^\alpha=\varphi^\alpha-\sum_{i=1}^p\xi^i u^\alpha_{i},
        \end{gathered}
\end{equation}
where the functions $Q_Y^\alpha$ are the so-called {\it characteristics} of the vector field $Y$. Geometrically, 
$$
Q^\alpha_Y=\iota_{{\bf j}^n\,Y}\theta^\alpha,\qquad \alpha=1,\ldots,q,
$$
where $\iota_Z\Upsilon$ stands for the contraction of a one-form $\Upsilon$ with a vector field $Z$.

First-order systems of PDEs in $p$ independent and $q$ dependent variables are defined by 
\begin{equation}
        \Delta^\mu({\bf j}^n_x\sigma)=0,\qquad \mu=1,\ldots,s\label{PDE},
\end{equation}
where $\Delta^\mu:J^n\rightarrow \mathbb{R}$ for $\mu=1,\ldots,s$ are certain functions. A particular solution of (\ref{PDE}) is a map $u(x)$ from $X$ to $U$ whose associated section $\sigma(x)=(x,u(x))\in \Gamma(\pi^0_{-1})$ is such that its prolongation to $J^n$  satisfies (\ref{PDE}). 

The system of PDEs (\ref{PDE}) determines a region ${\mathcal{S}_\Delta}\subset J^n$ where all the functions $\Delta^\mu$, with $\mu=1,\ldots,s$, vanish simultaneously. As is standard in the literature \cite{Ol93}, it is hereafter assumed that system (\ref{PDE}) has {\it maximal rank}, i.e. the functions $\Delta^\mu$ are functionally independent and  ${\mathcal{S}_\Delta}$ can be considered as a submanifold of $J^n$ (cf. \cite[p. 158]{Ol93}). A system of PDEs (\ref{PDE}) is {\it locally solvable} if for each ${\bf j}^n_x\sigma\in {\mathcal{S}_\Delta}$ there exists a solution $u(x)$ of the system (\ref{PDE}) such that ${\bf j}^n_x\sigma$ belongs to ${\bf j}^n\sigma_u(x)$, with $\sigma_u(x)=(x,u(x))$ \cite[p. 158]{Ol93}. Finally, we also assume that $\pi^n_{-1}(S_\Delta)=X$, which ensures that there exists at least one solution of $\mathcal{S}_\Delta$ defined around every $x\in X$.

A large family of differential equations is locally solvable and has maximal rank. To understand the generality of the local solvability assumption, we provide the following proposition characterising local solvability of first-order systems of PDEs in normal form. Similar results, with fewer details, can be found in \cite{GHZ05,Gu08}.

\begin{prop}\label{ZCC-LS} A first-order system of PDEs in normal form 
\begin{equation}\label{NormalSys}
\frac{\partial u^\alpha}{\partial x^i}=\phi^\alpha_i(x,u),\qquad i=1,\ldots,p,\quad \alpha=1,\ldots,q,
\end{equation}
is locally solvable if and only if it is integrable, i.e. $D_j\phi^\alpha_i=D_i\phi^\alpha_j$ for every $\alpha=1,\ldots,q$ and $i,j=1,\ldots,p$ \cite{Ol93}. In turn, its integrability amounts to the fact that the Lie algebra of vector fields spanned by 
	\begin{equation}\label{B}
Z_j=\frac{\partial}{\partial x^j}+\sum_{\alpha=1}^q\phi^\alpha_j\frac{\partial}{\partial u^\alpha},\qquad j=1,\ldots,p.
\end{equation}	
is Abelian.
\end{prop}
\begin{proof} Geometrically, the system (\ref{NormalSys}) amounts to a submanifold $\mathcal{S}_\Delta$ given by points of the form $(x^i,u^\alpha,\phi^\alpha_i(x,u))$ of $J^1$. The system (\ref{NormalSys}) has a particular solution passing through each point $(x_0^i,u_0^\alpha,\phi^\alpha_i(x_0,u_0))$ in $\mathcal{S}_\Delta$ if and only if the system (\ref{NormalSys}) admits a particular solution $u(x)$ for each initial condition $(x_0,u_0)\in X\times U$. If such a particular solution exists, then the cross partial derivatives of $u(x)$ relative to any two independent variables coincide. Using (\ref{NormalSys}), we obtain that
	\begin{equation}\label{DiffCon}
	\phi^\alpha_{j,k}-\phi^\alpha_{k,j}+\sum_{\beta=1}^q(\phi^\alpha_{j,u^\beta}\phi^\beta_{k}-\phi^\alpha_{k,u^\beta}\phi^\beta_j)=0,\qquad  j,k=1,\ldots,p,\quad\alpha=1,\ldots,q,
	\end{equation}
      The latter amounts to the condition $D_j\phi^\alpha_i=D_i\phi^\alpha_j$ for every $\alpha=1,\ldots,q$ and $i=1,\ldots,p$. Hence, system (\ref{NormalSys}) is integrable. 

      Note that the integrability of system (\ref{NormalSys}), namely (\ref{DiffCon}), amounts to 
      	$$
	[Z_k,Z_j]=\sum_{\alpha=1}^q\left(\phi^\alpha_{j,k}-\phi^\alpha_{k,j}+\sum_{\beta=1}^q(\phi^\alpha_{j,u^\beta}\phi^\beta_{k}-\phi^\alpha_{k,u^\beta}\phi^\beta_j)\right)\frac{\partial}{\partial u^\alpha}=0.
	$$
	Hence, if (\ref{NormalSys}) is integrable, then the distribution spanned by $Z_1,\ldots,Z_r$ is integrable and its integral submanifolds give solutions of (\ref{NormalSys}) for every point $(x_0,u_0)\in J^0$. This implies that (\ref{NormalSys}) is locally solvable.
%The system (\ref{NormalSys}) is integrable if and only if it admits a solution for every $(x_0,u_0)\in X\times U$. This finishes the first part of the proof.
%	
%	Moreover, observe that the integrability condition for (\ref{NormalSys}) reads
%	while 
%	This shows that (\ref{NormalSys}) is integrable if and only if the vector fields (\ref{B}) span an Abelian Lie algebra.
\end{proof}

A vector field $Y$ on $J^0$ is a {\it Lie point symmetry} of the system of PDEs (\ref{PDE}) if 
\begin{equation}\label{Sys}
        ({\bf j}^n\,Y)\,\Delta^\mu|_{\mathcal{S}_\Delta}=0, \qquad \mu=1,\ldots,s.
\end{equation}
As the prolongation to $J^n$ of the Lie bracket of two vector fields on $X\times U$ is the Lie bracket of their prolongations to $J^n$ \cite{Ol93}, the Lie bracket of two Lie point symmetries of $\Delta$ is a new Lie point symmetry. Thus, Lie point symmetries form a Lie algebra $V_s$, which, when finite-dimensional, locally defines a Lie group action on $J^0$. This Lie group action transforms solutions of (\ref{PDE}) into solutions of the same equation and allows for the reduction of the initial systems of PDEs \cite{St89}.

\section{A definition of characteristic systems for linear spaces of vector fields}

The description of Lie algebras of conditional symmetries for systems of PDEs is based on the geometry of the {\it characteristic system}, namely a submanifold of $J^n$ to which we restrict the study of solutions of our initial system of PDEs \cite{GMR97,GR95,OV95,Zi93}. As shown in this section, the literature on conditional symmetries lacks a purely geometric definition of characteristic systems and the study of characteristics systems for higher-order PDEs is scarcely considered \cite{GMR97,GR95,Ol92,Zi93}. This section aims to fill this gap. Moreover, our definition is not limited to characteristic systems for Lie algebras of vector fields \cite{Ol92,Zi93}, which will allow us to dispense, in the next sections, with certain unnecessary technical conditions on characteristic systems given in previous works \cite{GMR97,GR95,Zi93}.

\begin{defn}\label{CharSys}
Let $L$ be a finite-dimensional linear space of vector fields on $X\times U$. The {\it characteristic system of $L$} in $J^n$ is the subset of $J^n$  given by
$$
\mathcal{S}^n_L=\{{\bf j}^n_x\sigma\in J^n:[\iota_{{\bf j}^n\,Y}\theta]({\bf j}^n_x\sigma)=0,\forall Y\in L,\forall \theta \in \mathcal{CF}(J^n)\},
$$
where $\mathcal{CF}(J^n)$ is the space of contact forms on $J^n$.
\end{defn}

Let us describe Definition \ref{CharSys} in coordinates in order to compare it with previous definitions \cite{GMR97,GR95,Ol92,Zi93} and to justify the term `characteristic system'.

 If $L$ admits a basis (as a linear space) of the form $Y_j=\sum_{k=1}^p\xi_j^k\partial/\partial x^k+\sum_{\alpha=1}^q\varphi^\alpha_j\partial/\partial x^\alpha$, with $j=1,\ldots,l$, then the basis $\theta_K^\alpha$ of $\mathcal{CF}(J^n)$, with $\alpha=1,\ldots, q$ and $|K|\leq n-1$, allows us, by using (\ref{Prol}), to write that
$$
\begin{gathered}
\iota_{{\bf j}^nY_j}\theta^\alpha=\varphi^\alpha_j-\sum_{i=1}^pu^\alpha_i\xi_j^i,\\
\iota_{{\bf j}^nY_j}\theta_K^\alpha=\psi_{Y_jK}^\alpha-\sum_{i=1}^pu^\alpha_{K,i}\xi_j^i=D_K\left(\varphi^\alpha_j-\sum_{i=1}^pu^\alpha_i\xi_j^i\right),
\end{gathered}
$$
for $\alpha=1,\ldots,q,j=1,\ldots,l,$ and $1\leq |K|\leq n-1$.
Hence, the functions $\iota_{{\bf j}^nY_j}\theta^\alpha$, where $\alpha=1,\ldots,q$, are the characteristics $Q^\alpha_{Y_j}$ of the vector field $Y_j$ and the $\iota_{{\bf j}^nY_j}\theta^\alpha_K=D_KQ^\alpha_{Y_j}$ are all the compositions of up to $n-1$ total derivatives of the characteristics $Q^\alpha_j$ of $Y_j$. Then, ${\bf j}^n_x\sigma\in \mathcal{S}_L^n$ if and only if $[\iota_{{\bf j}^n\,Y_j}\theta_K^\alpha]({\bf j}^n_x\sigma)=0$ for $j=1,\ldots,l$ and the basis of basic contact forms $\theta^\alpha_K$ on $J^n$ with $\alpha=1,\ldots,q$ and $0\leq |K|\leq n-1$. Consequently, $\mathcal{S}^n_L$ is the subset of $J^n$ where vanish all total derivatives up to order $n-1$ of the characteristics of a basis of vector fields of $L$. 

In the theory of conditional symmetries, a characteristic system in $J^n$ is mostly defined to be the subset of zeros of the characteristics (and their total derivatives up to order $n-1$) of a basis of a Lie algebra of vector fields \cite{GMR97,GR95,Ol92,OV95,Zi93}. Hence, Definition \ref{CharSys} reproduces the standard definition when $L$ is a Lie algebra. Our definition is purely geometrical as it does not rely on the use of characteristics of a basis of $L$ and it does not need to give their coordinate expressions as in previous works.

Some definitions of characteristic systems demand additional technical conditions on the elements of $L$ (see \cite{Zi93}). The reasons to assume these conditions will be explained in the following sections. Meanwhile, Olver  and Roseneau proposed a very general definition of  differential constraints, the so-called {\it side conditions},  for studying higher-order systems of PDE \cite{Ol92,OR87}. This generalisation covers Definition \ref{CharSys} as a particular case. Nevertheless, Olver and Roseneau's definition is not linked to characteristics of vector fields due to its generality  \cite[p. 15]{Ol92}, which makes it inappropriate to the study of conditional symmetries, which are strongly related to characteristics.

Note that since $\mathcal{S}^n_L\subset J^n$ can be understood as the $n$-th order system of PDEs determined by the characteristics and their successive total differentials up to order $n-1$ of the elements of a linear space of vector fields $L$, it makes sense to call $\mathcal{S}^n_L$ a characteristic system.  

As the Cartan distribution $\mathcal{C}^{(n)}$ is the intersection of all the kernels of contact forms on $J^n$, the definition of $\mathcal{S}^n_L$ can be rewritten in the following dual equivalent manner.
\begin{defn}\label{CharSysII}
	Let $L$ be a finite-dimensional linear space of vector fields on $X\times U$. The {\it characteristic system of $L$} in $J^n$ is the subset $\mathcal{S}^n_L\subset J^n$ where the prolongations to $J^n$ of the vector fields of $L$ take values in the Cartan distribution $\mathcal{C}^{(n)}$.

\end{defn}

\section{On characteristic systems and sections of jet bundles}

In the Clairin approach to the theory of conditional symmetries \cite{Cl03,GMR97,GR95}, characteristic systems are employed to study systems of PDEs in normal form. This section studies this relation in geometric terms.

Note that $\mathcal{S}_L^n$ does not need to be a submanifold of $J^n$. For instance, if $X=\mathbb{R},U=\mathbb{R}$ and $L=\langle x_1(u_1^2-u_1)\partial_{u^1}\rangle$, then $\mathcal{S}^1_L$ is not a submanifold of $J^1$.  Nevertheless, in the next sections, ${\mathcal{S}^n_L}$ will be proven to be a submanifold due to the nature of the Clairin theory of conditional symmetries \cite{Cl03,GMR97,GR95}. The following results will explain the geometrical meaning of such assumptions and specify which of them can be dismissed. We begin with the following lemma.

\begin{lem} \label{CharacLin} Let $Y_1,\ldots, Y_l$ be a basis of a linear space $L$ of vector fields on $X\times  U$. If $f^1,\ldots,f^l\in C^\infty(X\times U)$, then one has the following equality at points of  $\mathcal{S}^n_L$:
	\begin{equation}\label{eq}
	{\bf j}^n\left(\sum_{j=1}^lf^jY_j\right)=      \sum_{j=1}^lf^j{\bf j}^nY_j,\qquad \alpha=1,\ldots,q,
	\end{equation}
	where the functions $f^1,\ldots, f^l$ on the right-hand side are considered to be functions on $J^n$ in the natural way, namely as functions on $J^n$ depending only on $x$ and $u$ (see \cite{VK99} for further details).
\end{lem}
\begin{proof} In coordinates we write
	\begin{equation}\label{basisL}
	Y_j=\sum_{i=1}^{p}\xi_j^i\frac{\partial}{\partial x^i}+\sum_{\alpha=1}^{q}\varphi_j^\alpha\frac{\partial}{\partial
		u^\alpha}
	\end{equation} 
	for $j=1,\ldots,l$. To simplify the notation, we write $\bar Y=\sum_{j=1}^lf^jY_j$.	In view of (\ref{Prol}), one has that 
	$
	Q^\alpha_{\bar Y}=\sum_{j=1}^lf^jQ^\alpha_{Y_j}.
	$
	Since all total derivatives $D_K$ of order $|K|\leq n-1$ of the characteristics of $Y_1,\ldots,Y_l$ vanish on $\mathcal{S}^n_L$, we obtain that 
	$
	D_KQ^\alpha_{\bar Y}=\sum_{j=1}^lf^jD_KQ^\alpha_{Y_j}
	$ for $|K|\leq n-1$ and $\alpha=1,\ldots,q$. 
	In view of (\ref{Prol}), one has that $\psi^\alpha_{\bar YK}=\sum_{j=1}^lf^j\psi^\alpha_{Y_jK}$, for  $\alpha=1,\ldots,q$, $|K|\leq n-1$, and formula (\ref{eq}) follows easily from the expression (\ref{Prol}) for the prolongations to $J^n$ of  $Y_1,\ldots,Y_l$. 
\end{proof}
We have already explained that $\mathcal{S}^n_L$ amounts to an $n$-th order system of PDEs. We are now concerned with establishing when such a system is in normal form. We prove the following rather simple fact to stress that $\mathcal{S}^n_L$ is a section of $J^n$ if and only if its associated $n$-th order system of PDEs can be written in normal form.
\begin{prop} A characteristic system $\mathcal{S}^n_L$ amounts to an $n$-th order system of PDEs in normal form if and only if $\mathcal{S}^n_L$ is a section of $J^n$.
\end{prop}
\begin{proof} If $\mathcal{S}^n_L$ is a section of $J^n$, then every coordinate $u^\alpha_K$ of points in $\mathcal{S}_L^n$ can be written as a function $u^\alpha_K=\phi^\alpha_K(x,u)$. Therefore, $\mathcal{S}_L^n$ can be written as the subset of $J^n$ where the series of conditions $u^\alpha_K-\phi^\alpha_K(x,u)=0$, with $\alpha=1,\ldots,q$ and $|K|\leq n-1$ are obeyed. Such conditions amount to an $n$-th order system of PDEs in normal form. The converse statement is immediate.
\end{proof}

Let us characterise, via the space $L$, when $\mathcal{S}_L^n$ amounts to an $n$-th order system of PDEs in normal form.

\begin{thm}\label{Th:SecSL1} Let $L$ be a linear space of vector fields on $X\times U$. Then, ${\mathcal{S}^1_L}$ is a section of the bundle $\pi^1_{0}:J^1\rightarrow J^0$ if and only if the vector fields of $L$ span a regular distribution $\mathcal{D}^L$ of rank $p$ and $\mathcal{D}^L$ projects onto $TX$ under $\pi^0_{-1}$.
\end{thm}
\begin{proof} 
 The linear  space $L$ admits a basis $Y_1,\ldots,Y_l$ of the form (\ref{basisL}).
 
 Let us prove first the converse part of our theorem. If $\mathcal{D}^L$ is regular of rank $p$ and its projection onto $X$  is $TX$ (relative to $\pi^0_{-1}$), then $Y_1,\ldots,Y_l$ span $\mathcal{D}^L_{(x,u)}\subset T(X\times U)$ and
        $$
        Y^H_j(x,u)=\sum_{i=1}^p\xi_j^i(x,u)\frac{\partial}{\partial x^i},\qquad j=1,\ldots,l=\dim L,
        $$
        span $T_xX$ for every $u\in U$. Consequently, we can assume without loss of generality that the first vector fields  $Y_1,\ldots, Y_p$ from $Y_1\ldots,Y_l$ are such that $Y_1(x,u),\ldots, Y_p(x,u)$ are linearly independent and span $\mathcal{D}^L$ at a point $(x,u)\in X\times U$. If $\xi$ is the $p\times p$  matrix with coefficients $\xi_j^i$ of $Y_1,\ldots Y_p$, then $\det \xi\neq 0$ on an open subset of $(x,u)$ and $\xi$ has an inverse $\xi^{-1}$. The space ${\mathcal{S}^1_L}$ is given by the zeroes of the functions
        $\iota_{{\bf j}^1\,Y_j}\theta^\alpha=\varphi_j^\alpha(x,u)-\sum_{i=1}^{p}\xi_j^i(x,u) u^\alpha_i$, with $j=1,\ldots,l$. Since $Y_1,\ldots,Y_p$ generate the distribution $\mathcal{D}^L$, in view of Lemma \ref{CharacLin}, the $\mathcal{S}_L^1$ is locally given by the zeroes of the functions $\iota_{{\bf j}^1\,Y_j}\theta^\alpha=\varphi_j^\alpha(x,u)-\sum_{i=1}^{p}\xi_j^i(x,u) u^\alpha_i$, with $j=1,\ldots,p$.
This shows that
        \begin{equation}\label{Sys1O}
        u^\alpha_k=\sum_{i=1}^p(\xi^{-1})^i_k(x,u)\varphi^\alpha_i(x,u),\qquad k=1,\ldots,p,\quad \alpha=1,\ldots,q.
        \end{equation}
        Thus, the value of $u^\alpha_k$ is determined for every $(x,u)$ univocally in terms of the $\varphi^\alpha_j$ and $\xi$. Hence, ${\mathcal{S}^1_L}$ is locally a section of $\pi^1_0$. The smoothness of the vector fields $Y_1,\ldots,Y_p$ ensures that ${\mathcal{S}^1_L}$ is locally a smooth section of $\pi^1_0$. As the same procedure can be applied to any point $(x,u)$, one easily finds that $\mathcal{S}^1_L$ is a global section of $\pi^1_0$.
        
        Let us prove the direct part of this theorem. If ${\mathcal{S}^1_L}$ is a smooth section of $\pi^1_0$, then the points $(x^i,u^\alpha,u^\alpha_i)$ of ${\mathcal{S}^1_L}$ are such that for every $x,u$, the possible values of $u^\alpha_k$ are univocally defined by the conditions
        $$
        \varphi_j^\alpha(x,u)-\sum_{i=1}^{p}\xi_j^i(x,u) u^\alpha_i=0,\qquad j=1,\ldots,l. 
        $$
        The values of $u^\alpha_i$ for every $x,u$ are univocally determined if and only if the extended and the standard matrix of coefficients of the previous system have rank $p$.  The matrix of coefficients has rank $p$ if and only if the projections of the vector fields of $L$ onto $X$ span a distribution of rank $p$, namely $TX$. As the extended matrix also has  rank $p$,  the vector fields $Y_1,\ldots, Y_l$ also span a regular distribution of rank $p$ on $J^0$.
\end{proof}
\begin{ex} Theorem \ref{Th:SecSL1} shows that $\mathcal{S}^1_L$ may be a section of $\pi_0^1$ even if $L$ is not a Lie algebra and/or its elements do not span an integrable distribution, these conditions being the standard used in other works \cite{GMR97,GR95}. For instance, consider  $U=\mathbb{R}$, $X=\mathbb{R}^2$ and $L=\langle \partial/\partial x+\partial/\partial u,\partial/\partial t+x\partial/\partial u\rangle$, which is not a Lie algebra. Then $\mathcal{S}_L^1$ is given by the zeroes of the conditions $u_x-1=0, u_t-x=0$, which amount to  a section on $\pi^1_0$.
\end{ex}

\begin{ex} Theorem \ref{Th:SecSL1} states that $\mathcal{S}^1_L$ may be a section of $\pi^1_0$ when $L$ is a Lie algebra of dimension different from $\dim X$. This case does not appear in the Clairin approach, where $\dim L=\dim X$, (see \cite{Cl03,GMR97,GR95}). To illustrate such a new possibility, consider the case where $U=\mathbb{R}$ and $X=\mathbb{R}^2$. The Lie algebra $L=\langle \partial/\partial x+\partial/\partial u,\partial/\partial t+\partial/\partial u,x(\partial/\partial t+\partial/\partial u)\rangle$ is a three-dimensional Lie algebra giving rise to a section on $\pi_0^1$ given by the equations $u_x-1=0,u_t-1=0,xu_t-x=0$. 
	\end{ex}

\begin{ex}
Let us provide an example of when a Lie algebra $L$ on $X\times U$ does not give rise to a section of $\pi^1_{-1}$. Consider $U=\mathbb{R}$, $X=\mathbb{R}$, and the Lie algebra $L=\langle x\partial/\partial x+\partial/\partial u\rangle$. The Lie algebra $L$ gives rise to a regular distribution of rank one, but the distribution does not project onto $T\mathbb{R}$ under $\pi^0_{-1}$ (it has a problem at points of $U\times X$ with $x=0$). Then, Theorem \ref{Th:SecSL1} ensures that $\mathcal{S}_L$ is not a section of $\pi^0_{-1}$. In fact, the characteristic system of $L$ is given by  $xu_x-1=0$ and the value of $u_x$ at $x=0$ is not determined. Thus, $\mathcal{S}_L$ is not a section of $\pi^1_{-1}$.
\end{ex}

 Let us now extend Proposition \ref{Th:SecSL1} to $J^n$. This will lead to the requirement that the distribution spanned by the elements of $L$ be involutive in cases relevant to us, namely where we are studying systems of PDEs and therefore $\dim X>1$. Before proving this result, we need the following lemma.
 
\begin{lem}\label{Lem:LZ} If $L$ gives rise to a section $\mathcal{S}^1_L$ of $\pi^1_0$, then the first-order system of PDEs (\ref{NormalSys}) associated with $\mathcal{S}^1_L$ is such that the vector fields	(\ref{B})
 span the same distribution as $L$.
\end{lem}
\begin{proof} Let $Y_1,\ldots,Y_l$ be a basis of $L$ given by (\ref{basisL}). In view of Theorem \ref{Th:SecSL1}, the distribution generated by the elements of $L$ has rank $p$ and there exist $p$ vector fields among $Y_1,\ldots,Y_l$, let us say without loss of generality that these are $Y_1,\ldots,Y_p$, whose coefficients allow us to determine a first-order system of PDEs given by (\ref{Sys1O}). Then, the vector fields 
	$$
	\frac{\partial}{\partial x^j}+\sum_{\alpha=1}^q\sum_{i=1}^p(\xi^{-1})^i_j(x,u)\varphi^\alpha_i\frac{\partial}{\partial u^\alpha}=\sum_{k=1}^p(\xi^{-1})_j^kY_k,\qquad j=1,\ldots,p
	$$
	span the same distribution as the vector fields of $L$.
\end{proof}

\begin{thm}\label{Th:SecSL} Let $L$ be a linear space of vector fields on $X\times U$. Then, ${\mathcal{S}^n_L}$, with $n>1$, is a smooth section of the bundle $\pi^n_{0}:J^n\rightarrow J^0$ if and only if the vector fields of $L$ span an involutive distribution $\mathcal{D}^L$ of rank $p$ projecting onto $TX$ under $\pi^0_{-1}$.
\end{thm}
\begin{proof}
	
	Let us prove the converse part of the theorem by induction relative to $n$. Assume that $Y_1,\ldots,Y_l$ given by (\ref{basisL}) is a basis of the linear space $L$.	
	In view of Theorem \ref{Th:SecSL1} and the considered assumptions,  $\mathcal{S}^{1}_L$ is a section of $\pi^1_0$ and then  $u^\alpha_j=\phi^\alpha_j(x,u)$
	for certain functions $\phi^\alpha_j$ on $\mathcal{S}^1_L$, with $\alpha=1,\ldots,q$ and $j=1,\ldots,p$. 
	Let us prove that if the $u^{\alpha}_R$, where $R$ is any multi-index with $|R|\leq h\leq n-2$ for a natural number $h$, can be written as functions depending on $x,u$ only on the subset $\mathcal{S}^n_L$, then the $u^\alpha_{K}$ for $|K|=h+1$ can also. Recall that one
	can guarantee on $\mathcal{S}^n_L$
	that
	\begin{equation}\label{eqa}
	0=\iota_{{\bf j}^nY_j}\theta^\alpha_R=D_R\left(\varphi^\alpha-\sum_{i=1}^p\xi_j^iu_{i}^\alpha\right),\quad \alpha=1,\ldots,q,\,\, j=1,\ldots,l,\,\, |R|\leq h.
	\end{equation}
	By our induction hypothesis, the coordinates $u_R^\alpha$ of points of $\mathcal{S}^n_L$  for $|R|\leq h$
can be written as functions of $x,u$ only, namely $u_R^\alpha=u_R^\alpha(x,u)$. Hence, rewriting the right-hand side of (\ref{eqa}), we obtain that
	\begin{equation}\label{eee}
	0=-\sum_{i=1}^pu^\alpha_{R,i}\xi^i_j(x,u)+F^\alpha_{Rj}(x,u),
	\end{equation}
	for certain functions $F^\alpha_{Rj}(x,u)$ with $\alpha=1,\ldots,q$, $j=1,\ldots,l$, and $|R|\leq h$ that gather all the terms of $D_R(\varphi^\alpha-\sum_{i=1}^p\xi_j^iu_i^\alpha)$ with derivatives of the $u^\alpha$ up to order $h$.
	Since $\mathcal{S}_L^1$ is a section of $\pi^1_0$, Theorem \ref{Th:SecSL1} ensures that there exist $p$ elements of the basis of $L$, which are assumed without loss of generality to be the first $p$ ones, such that their coordinate functions $\xi^k_j$, for $i,j=1,\ldots,p$, are the entries of an invertible matrix $\xi$. Then, the expressions (\ref{eee})
	show that the coordinates $u^\alpha_{R,k}$ of the points of $\mathcal{S}_L^n$ are determined by a system of algebraic equations depending only on $x,u$. Hence, if the system (\ref{eee}) admits a solution, then it is unique. 
	
	Let us prove that the system (\ref{eee}) is compatible. If $\dim X=1$, then this is obvious as the expressions (\ref{eee}) determine each derivative of $u^\alpha$ in terms of the unique independent variable uniquely. If $\dim X>1$, then $u^\alpha_{i,j}=u^\alpha_{j,i}$ for $i\neq j$ but the expressions for $u^\alpha_{i,j}$ and $u^\alpha_{j,i}$ obtained from  (\ref{eee}) may give different values at a point $(x,u)$. To prove that system (\ref{eee}) is compatible, recall that  the elements of $L$ span an involutive distribution. 
	Hence, Lemma \ref{Lem:LZ} ensures that the vector fields $Z_j$ span an involutive distribution and the associated first-order system of PDEs in normal form induced by $L$, let us say (\ref{NormalSys}),
	admits a local solution $u(x)$ for every initial condition $(x_0,u_0)$. On the sections ${\bf j}^n\sigma$, for $\sigma(x)=(x,u(x))$, the characteristics of all $Y_1,\ldots,Y_{l}$, and their total derivatives vanish identically. Hence, the pull-back of expressions (\ref{eee}) relative to ${\bf j}^n\sigma$ are zero and the system (\ref{eee}) is compatible.  By the induction hypothesis, $\mathcal{S}^n_L$ is a section relative to $\pi^n_0$.
	
		Let us prove the  direct part of the theorem. If $\mathcal{S}^n_L$ is a section of $\pi^n_0$, then $\mathcal{S}^{1}_L$ is a section of $\pi^1_0$ and, using Theorem \ref{Th:SecSL1}, we obtain that $\mathcal{D}^L$ has rank $p$ and projects onto $TX$ under $\pi^0_{-1}$. It is left to prove that $\mathcal{D}^L$ is involutive. If $\dim X=1$, the result is immediate. If $\dim X>1$, we can use the same arguments as in the proof of Theorem \ref{Th:SecSL1} to obtain that the first-order system of PDEs associated with $\mathcal{S}^1_L$ reads (\ref{Sys1O}). In view of Lemma \ref{CharacLin}, this system is given by the zeroes of the characteristics
\begin{displaymath}\sum_{i=1}^q(\xi^{-1})^i_k\varphi^\alpha_i-u^\alpha_k=\sum_{j=1}^q(\xi^{-1})^{j}_k\iota_{{\bf j}^nY_j}\theta^\alpha=\iota_{{\bf j}^n\sum_{j=1}^q(\xi^{-1})^{j}_kY_j}\theta^\alpha=0\end{displaymath}
for $k=1,\ldots,p$ and  $\alpha=1,\ldots,q$. Since $\mathcal{S}^2_L$ is a section, one has that the equations $D_i\iota_{{\bf j}^nY_j}\theta^\alpha=0$ and $D_j\iota_{{\bf j}^nY_i}\theta^\alpha=0$, with $i\neq j$, must lead to a unique solution for $u^\alpha_{i,j}$ at every $x,u$. In terms of the Lemma \ref{CharacLin}, one has that $D_i(\sum_{j=1}^q(\xi^{-1})^j_k\varphi^\alpha_j-u^\alpha_k)=D_k(\sum_{j=1}^q(\xi^{-1})^j_i\varphi^\alpha_j-u^\alpha_i)$, which implies that the vector fields 
		\begin{equation}\label{eq:Bas}
	\frac{\partial}{\partial x^j}+\sum_{\alpha=1}^q\sum_{j=1}^q(\xi^{-1})^j_k(x,u)\varphi^\beta_j\frac{\partial}{\partial u^\alpha},\qquad j=1,\ldots,p
	\end{equation}
	commute among themselves. In view of Lemma \ref{Lem:LZ}, the vector fields of $L$ are linear combinations with functions in $C^\infty(U\times X)$ of the vector fields (\ref{eq:Bas}). Then, they span an involutive distribution.  
		
	\end{proof}

\begin{rem}\label{Uni}The proof of  Theorem \ref{Th:SecSL} shows that  equations (\ref{eee}) determine the value of each coordinate $u^\alpha_K$, with $|K|\geq 1$, of points of a section $\mathcal{S}^n_L$ of $\pi^n_0$ with $n=2,3,4,\ldots$, in terms of the $u^1_R,\ldots u^q_R$ with $|R|<|K|$. 
	Thus, two integrable sections $\mathcal{S}^n_L$ and $\mathcal{S}^n_{L'}$, with $n>1$, sharing the same projection onto $J^1$ are the same. Moreover, solutions of the system of PDEs $\mathcal{S}^n_L$, with $n>1$, must be holonomic sections. Additionally, Proposition \ref{Prop:Coi} will permit us to determine the necessary and sufficient conditions on $L$ and $L'$ to ensure that their characteristic systems are sections of $\pi^n_0$ such that $\mathcal{S}^n_L=\mathcal{S}^n_{L'}$.
\end{rem}

\section{On the geometry of characteristic systems}

It is frequently assumed in the Clairin theory of  conditional symmetries that $L$ is an Abelian Lie algebra of dimension $\dim X$ admitting a basis of a particular type \cite{Cl03,GMR97,GR95}. This ensures that $\mathcal{S}_L^n$ is a section of $\pi_0^n$ amounting to an integrable system of PDEs and the vector fields of $L$ are tangent to $\mathcal{S}_L^n$. This section provides necessary and sufficient conditions on $L$ to ensure previous results. More particularly, we will find that the conditions appearing in \cite{GMR97,GR95} can be significantly relaxed.

\begin{prop}\label{Prop:Coi} Let $L$ and $L'$ be linear spaces of vector fields on $X\times U$ whose characteristic systems in $J^n$ are sections of $\pi^n_0$. Then,  $\mathcal{S}^n_L=\mathcal{S}^n_{L'}$ if and only if the vector fields of $L$ and $L'$ span the same distribution.
	\end{prop}
\begin{proof} Let us prove the direct part. Since $\mathcal{S}^n_L$ is a section of $\pi^n_0$, the projection of $\mathcal{S}^n_L=\mathcal{S}^n_{L'}$ to $J^1$ via $\pi^n_1$ gives rise to a section of $\pi^1_0$. Hence, $L$ and $L'$ determine the same first-order system of PDEs in normal form, let us say (\ref{NormalSys}). By virtue of Lemma \ref{Lem:LZ}, the distributions $\mathcal{D}^L$ and $\mathcal{D}^{L'}$ spanned by the elements $L$ and $L'$, respectively, coincide with the one generated by  (\ref{B}).
	Hence, $\mathcal{D}^L=\mathcal{D}^{L'}$.
	
	Conversely, if $L$ and $L'$ span the same distribution, and since $\mathcal{S}^n_L$ and $\mathcal{S}^n_{L'}$ are sections of $\pi^n_0$, then the ranks of $\mathcal{D}^L$, $\mathcal{D}^{L'}$ are, by virtue of Theorems  \ref{Th:SecSL1} and \ref{Th:SecSL}, equal to $p$ and there exists a family of functions $\xi^j_i$ on $J^0$ giving rise to an invertible $p\times p$ matrix $\xi$ mapping $p$ elements of $L$ spanning $\mathcal{D}^L$ onto $p$ elements of $L'$ spanning $\mathcal{D}^{L'}$. Consequently, the characteristics of the selected elements of $L$ and $L'$, let us say $Q^\alpha_j$ and $(Q')^\alpha_j$, with $\alpha=1,\ldots,q,j=1,\ldots,p$, satisfy $Q^\alpha_j=\sum_{k=1}^p\xi^k_j(Q')^\alpha_k$, with $i,j=1,\ldots,p$ and $\alpha=1,\ldots,q$. Therefore, it follows that $D_KQ^\alpha_j=\sum_{k=1}^p\xi^k_jD_K(Q')^\alpha_k$ on $\mathcal{S}_{L'}^n$ for every multi-index $K$ such that $|K|\leq n-1$. In view of Lemma \ref{CharacLin}, the same can be extended to arbitrary elements of $L$. Hence, $\mathcal{S}^n_{L'}\subset \mathcal{S}^n_L$. Since $\xi$ is invertible, one can repeat the above procedure considering that $(Q')^\alpha_j=\sum_{k=1}^p(\xi^{-1})^k_jQ_k^\alpha$ to obtain that $\mathcal{S}^n_{L}\subset \mathcal{S}^n_{L'}$ and $\mathcal{S}^n_L=\mathcal{S}^n_{L'}$.

	\end{proof}

\begin{lem}\label{Lem:InvAbe} If the vector fields $Z_1,\ldots,Z_p$ given by $(\ref{B})$ span an involutive distribution on $J^1$, then $L=\langle Z_1,\ldots,Z_p\rangle$ is an Abelian Lie algebra.
	\end{lem}
\begin{proof}
	By the involutivity assumption on the distribution spanned by the elements of $L$, we have that
	$$
	[Z_j,Z_k]=\sum_{m=1}^pf_{jk}^mZ_m,\qquad j,k=1,\ldots,p,
	$$  
	for certain functions $f^m_{jk}\in C^\infty(X\times U)$ with $j,k,m=1,\ldots,p$. Since the left-hand side projects onto zero relative to $\pi^0_{-1}$, the right-hand side does also. Hence,
	$$
	\sum_{m=1}^pf_{jk}^m(x,u)\frac{\partial}{\partial x^m}=0
	$$
	for every $(x,u)\in X\times U$. Therefore, $f_{jk}^m=0$ for all possible indices $j,k,m$ and the vector fields $Z_1,\ldots,Z_p$ are in involution.
	\end{proof}

\begin{thm}\label{Th:ProlSL} Let  $\mathcal{S}^n_L$ be a section of $\pi^n_{0}$. The prolongations to $J^n$ of the elements of $L$ are tangent to ${\mathcal{S}^n_L}$ if and only if the distribution $\mathcal{D}^L$ spanned by the elements of $L$ is involutive.
\end{thm}

\begin{proof} Assume first that $\mathcal{D}^L$ is involutive. The prolongations to $J^n$ of the elements of $L$ are tangent to ${\mathcal{S}^n_L}$ if and only if  the functions $\iota_{{\bf j}^nY}\theta$, where $Y$ is any element of $L$ and $\theta$ is any contact form on $J^n$, are first-integrals of any ${\bf j}^n\,Z$ with $Z\in L$. Now,
         \begin{equation}\label{Fun}
         {\bf j}^nZ(\iota_{{\bf j}^nY}\theta)=\iota_{{\bf j}^nY}[\mathcal{L}_{{\bf j}^nZ}\theta]+\iota_{{\bf j}^n[Z,Y]}\theta.
         \end{equation}
         Let us analyse both right-hand terms on $\mathcal{S}_L^n$. First,
        as ${\bf j}^n\,Z$ is a prolongation to $J^n$ of  $Z$, one has that $\mathcal{L}_{{\bf j}^nZ}\theta$ is a contact form and, by the definition of $\mathcal{S}^n_L$, one has
         \begin{equation}\label{con1A}
        \iota_{{\bf j}^nY}[ \mathcal{L}_{{\bf j}^nZ}\theta]|_{\mathcal{S}^n_L}=0.
         \end{equation}
         Second, we make the assumption that $\mathcal{S}_L^n$ is a section of $\pi^n_0$. Then Theorem \ref{Th:SecSL} states that $\mathcal{D}^L$ has order $p$. Since $\mathcal{D}^L$ is assumed to be involutive, we have  $[Z,Y]=\sum_{k=1}^{p}f^kY_k$ for  the functions $f^1,\ldots,f^p\in C^\infty(J^0)$ and some elements $Y_1,\ldots,Y_p$, with $Y_1\wedge\ldots\wedge Y_p\neq 0$, chosen from a basis of $L$. In view of Lemma \ref{CharacLin}, one gets that 
         \begin{equation}\label{con2A}
         \iota_{{\bf j}^n[Z,Y]}\theta|_{{\mathcal{S}^n_L}}=\iota_{{\bf j}^n\left[\sum_{k=1}^{p }f^kY_k\right]} \theta\big|_{{\mathcal{S}^n_L}}\!\!\!\!=\sum_{k=1}^{p }f^k\iota_{{\bf j}^nY_k}\theta\big|_{{\mathcal{S}^n_L}}=0.
         \end{equation}
         Using (\ref{con1A}) and (\ref{con2A}) to simplify (\ref{Fun}), we get that ${\bf j}^n\,Z[\iota_{{\bf j}^nY}\theta]|_{\mathcal{S}^n_L}=0$. Consequently, the prolongations to $J^n$ of the elements of $L$ are tangent to ${\mathcal{S}^n_L}$.

We now prove the converse by contradiction. Assume that  $
\mathcal{D}^L$ is not involutive and that the prolongations to $J^n$ of the elements of $L$ are tangent to $\mathcal{S}^n_L$. Then, there exist $Z,Y\in L$ such that $[Z,Y]$ does not take values in $\mathcal{D}^L$. Thus, the elements of the linear space $L'=L\oplus \langle [Z,Y]\rangle$ span a distribution $\mathcal{D}^{L'}$ of rank $p+1$. Since by assumption ${\bf j}^n\, Z [\iota_{{\bf j}^n\, Y}\theta]|_{\mathcal{S}^n_L}=0$ for any contact form on $J^n$, the equality (\ref{Fun}) shows that  $\iota_{{\bf j}^n[Z,Y]}\theta|_{{\mathcal{S}^n_L}}=0$, and that $\mathcal{S}^n_{L'}$ contains ${\mathcal{S}^n_L}$, namely $\mathcal{S}^n_{L}\subset \mathcal{S}^n_{L'}$. Since $L\subset L'$, one has that  $\mathcal{S}^n_{L'}\subset \mathcal{S}_L^n$ and therefore $\mathcal{S}^n_{L'}=\mathcal{S}^n_L$. Since $\mathcal{S}^n_{L'}$ is a section, Theorems \ref{Th:SecSL} and \ref{Th:SecSL1} state that $\mathcal{D}^{L'}$ has rank $p$. This is a contradiction and $\mathcal{D}^L$ must be involutive. 
         
        \end{proof}

The theory of conditional symmetries focuses on the case where $L$ is a Lie algebra. Then, $\mathcal{D}^L$ is involutive \cite{Palais} and, therefore, an immediate consequence of  Theorem \ref{Th:ProlSL} is the following trivial corollary. As we do not assume that $\mathcal{S}^n_L$ is a section of $\pi^n_0$,  the corollary can be applied to general conditional symmetries that need not be of the Clairin type \cite{Ol92}.

\begin{cor}\label{LTan} If $L$ is a Lie algebra of vector fields on $X\times U$, then the prolongations to $J^n$ of the elements of $L$ are tangent to ${\mathcal{S}^n_L}$.       
\end{cor}

%\section{Necessary and sufficient conditions on $L$ for the local solvability of $\mathcal{S}^n_L$}

Any section of $\pi^n_0$, let us say 
$$\sigma: (x,u)\in X\times U\mapsto (x^i,u^\alpha
,\phi^\alpha_i(x,u),\ldots,\phi^\alpha_{K_{n-1}}(x,u),\phi^\alpha_{K_n}(x,u))\in J^n,$$ 
where $|K_j|=j$ for $j=2,\ldots,n$,  leads to an $n$-th order system of PDEs in normal form
\begin{equation}\label{HOSys}
\frac{\partial u^\alpha}{\partial x^i}=\phi^\alpha_i(x,u),\ldots,\frac{\partial^{j} u^\alpha}{\partial x^{K_j}}=\phi^\alpha_{K_j}(x,u),
\end{equation}
for $\alpha=1,\ldots,q$, $i=1,\ldots,p$, $j=1,\ldots,n$,
and vice versa. This fact is frequently employed in the theory of conditional
symmetries, where $\mathcal{S}^n_L$ is constructed in such a way that it is a section of
$\pi^n_0$ and, consequently, amounts to a system of PDEs in normal form (cf. \cite{GMR97,GR95,OV95}).

From now on we assume that all characteristic systems $\mathcal{S}^n_L$  are sections of $\pi^n_0$. Theorem \ref{Th:SecSL} then ensures that the distribution $\mathcal{D}^L$ spanned by the elements of $L$ has rank $p$ and its projection to $X$ (via $\pi^0_{-1}$) is $TX$. Our main objective in this section is to show that the system of PDEs related to $\mathcal{S}^n_L$ is locally solvable if and only if $\mathcal{D}^L$ is an involutive distribution.

\begin{prop}\label{Prop:SLInt} Let $L$ be a linear space of vector fields whose  $\mathcal{S}^n_L$ is a section of $\pi^n_0$. Then, $\mathcal{S}^n_L$ is locally solvable if and only if the vector fields of $L$ span an involutive distribution $\mathcal{D}^L$.
\end{prop}
\begin{proof} Assume first that $\mathcal{D}^L$ is involutive. Theorem \ref{Th:ProlSL} ensures that $\mathcal{D}^L$ is tangent to $\mathcal{S}^n_L$. Given any point ${\bf j}^n_x\sigma\in \mathcal{S}^n_L$, there exists a unique integral submanifold  $\sigma^{(n)}$ of $\mathcal{D}^L$ passing through this point. Since $\mathcal{D}^L$ projects onto $TX$, this can be considered as a section of $\pi^n_0$. As the elements of $\mathcal{D}^L$ are tangent to $\mathcal{S}^n_L$, the section $\sigma^{(n)}$ is contained in  $\mathcal{S}^n_L$. Due to Remark \ref{Uni},  the section $\sigma^{(n)}$  is holonomic. Hence, it gives rise to a solution of $\mathcal{S}^n_L$, which becomes a locally solvable system of PDEs. 
	
	The projection of $\mathcal{S}^n_L$ to $J^1$ is given by the points of $J^1$ satisfying the condition that the characteristics of elements of $L$ vanish. This is the condition characterising $\mathcal{S}^1_L$. Hence, the projection of $\mathcal{S}^n_L$ to $J^1$ is $\mathcal{S}^1_L$.	If $\mathcal{S}^n_L$ is locally solvable, then its projection to $J^1$ via $\pi^n_1$, namely $\mathcal{S}^1_L$, is also locally solvable and, in view of Proposition \ref{ZCC-LS}, the compatibility conditions for this system are satisfied. The vector fields $Z_1,\ldots,Z_p$ of the form  (\ref{B})  describing $\mathcal{S}^1_n$ and those of $L$ span the same distribution by virtue of Lemma \ref{Lem:LZ}. Then, one has that $\mathcal{D}^L$ is involutive.

\end{proof}

The above proposition justifies the following definition.
\begin{defn} A section $\hat \sigma$ of $\pi^n_0$ is  {\it integrable} when its associated system of PDEs is locally solvable.
        \end{defn}

Finally, let us give a rather immediate consequence of Proposition \ref{Prop:SLInt} that will allow us to simplify the application of the Clairin theory of conditional symmetries.

\begin{cor}\label{Simplification} If $\mathcal{S}_L^n$ is integrable,  then every initial condition admits a unique solution. Moreover, the solutions of $\mathcal{S}^n_L$ and $\mathcal{S}^1_L$ are the same.
\end{cor}

\section{Generalising a standard assumption in the theory of conditional symmetries}\label{Sec:NPDEL}

The Clairin theory of conditional symmetries  \cite{Cl03} mainly focuses on the case where $L$ is a $p$-dimensional Lie algebra of conditional symmetries admitting a basis of a particular type \cite{Cl03,GMR97,GR95}. As far as we know, no work deals with more general types of $L$. This section is aimed at characterising this type of Lie algebra. We will also define a larger family of Lie algebras whose properties will make them more useful to study Lie algebras of conditional symmetries, as will be seen in the following sections.

\begin{prop}\label{LemBas} A Lie algebra $L$ of vector fields on $X\times U$ admits a basis of the form  (\ref{B}) if and only if $L$ is  a $p$-dimensional Abelian Lie algebra and the distribution $\mathcal{D}^L$ generated by the elements of $L$ projects onto $TX$ via $\pi^0_{-1}$.
\end{prop}
\begin{proof} Let us prove the direct part of the proposition.  If $L$ admits a basis $Z_1,\ldots,Z_p$ given by (\ref{B}), then $Z_1,\ldots,Z_p$, and therefore all the elements of $L$, are projectable onto $X$. The projections of $Z_1,\ldots,Z_p$ are $\partial/\partial x^k$, with $k=1,\ldots,p$, and therefore $\mathcal{D}^L$ projects onto $TX$ via $\pi^0_{-1}$. 
	
        Let us now prove the converse.  If $L$ is a Lie algebra of vector fields projectable onto $X$, then the elements of a basis $X_1,\ldots,X_p$ of $L$ are also projectable. Since $L$ is Abelian, the projections span an Abelian Lie algebra. Since $\mathcal{D}^L$ projects onto $TX$ under $\pi^0_{-1}$, the projections $\pi^0_{-1*}X_j$, with $j=1,\ldots,p$,  generate the distribution $TX$ and commute among themselves. Fr\"obenius' Theorem ensures that there exists a coordinate system $x^1,\ldots,x^{p}$ on $X$ such that 
        $$
        \pi^0_{-1*}X_j=\frac{\partial}{\partial x^j},\qquad j=1,\ldots,p.
        $$
        Adding to the coordinates $x^1,\ldots,x^p$ a new set of coordinates to form a coordinate system on $X\times U$, we obtain that the vector fields $X_1,\ldots,X_p$ can be written in the form 
        $$
        X_j=\frac{\partial}{\partial x^j}+\sum_{\alpha=1}^{q}\phi^\alpha_j\frac{\partial}{\partial u^\alpha},\qquad j=1,\ldots,p,
        $$
for certain functions $\phi^\alpha_j\in C^\infty(J^0)$, with $\alpha=1,\ldots,q$ and $j=1,\ldots,p$. This finishes the converse part of our proposition. 
\end{proof}

Due to their appearance in the literature (cf. \cite{GMR97,GR95,Ol92}) and in this work, the Lie algebras studied in the above proposition deserve a special name.

\begin{defn} A {\it rectified PDE Lie algebra} is a $p$-dimensional Abelian Lie algebra of vector fields on $X\times U$ whose projections onto $X$ span $TX$.
\end{defn}

 We aim to show in the following sections that one can significantly enlarge the Clairin theory of conditional symmetries by considering Lie algebras $L$ of conditional symmetries  giving a section $\mathcal{S}^n_L$ of $\pi^n_0$  such that there exists a rectified PDE Lie algebra $L'$ satisfying  $\mathcal{S}^n_L=\mathcal{S}^n_{L'}$. The following lemma is key to accomplishing this goal. 

\begin{lem}\label{Renormalization} Let $L$ be a linear space of vector fields on $X\times U$ whose $\mathcal{S}^n_L$ is an integrable section of $\pi^n_0$ and such that the distribution $\mathcal{D}^L$ has rank $p$ and projects onto $TX$ under $\pi^0_{-1}$. Then, there exists a rectified PDE Lie algebra $L'$ such that $\mathcal{S}^n_L=\mathcal{S}^n_{L'}$.
        \end{lem}
\begin{proof} Under the considered assumptions on $\mathcal{D}^L$, one has that $L$ admits a family $Y_1,\ldots, Y_p$ of elements spanning $\mathcal{D}^L$ such that $Y_1\wedge \ldots \wedge Y_p\neq 0$. Assume that the vector fields $Y_1,\ldots, Y_p$ take the form (\ref{basisL}). Since $\mathcal{D}^L$ projects onto $TX$  relative to $\pi^0_{-1}$, one has that the functions $\xi^k_j$ for $j,k=1,\ldots,p$ are the coefficients of an invertible $p\times p$ matrix $\xi$. Since $\xi$ is invertible, one can define a new family of linearly independent vector fields
        \begin{equation}\label{renor}
        Z_k=\sum_{j=1}^p(\xi^{-1})^j_kY_j=\frac{\partial}{\partial x^k}+\sum_{j=1}^p\sum_{\alpha=1}^q(\xi^{-1})_k^j\varphi^\alpha_j\frac{\partial}{\partial u^\alpha},\quad k=1,\ldots,p.
        \end{equation}
        Since $\mathcal{D}^L$ coincides with the distribution spanned by the elements $L'=\langle Z_1,\ldots Z_p\rangle $, Lemma \ref{Lem:InvAbe} shows that $L'$ is an involutive Lie algebra and $L'$ becomes a rectified PDE Lie algebra. Since $\mathcal{D}^L=\mathcal{D}^{L'}$,  Proposition \ref{Prop:Coi} ensures that $\mathcal{S}^n_L=\mathcal{S}^n_{L'}$. 
        \end{proof}

Lemma \ref{Renormalization} and other results obtained in the following sections will allow us to extend findings concerning rectified PDE Lie algebras $L'$ of conditional symmetries to more general Lie algebras $L$.  This justifies the introduction of the following definition.

\begin{defn}\label{Def:RLS} A {\it rectifiable linear space of vector fields} $L$ is a linear space of
vector fields on $X\times U$ whose characteristic system is equal to the characteristic system of a rectified PDE Lie algebra $L'$. Then, $L'$ is called an {\it associated rectified PDE Lie algebra} of $L$. 
\end{defn}

\begin{ex} Assume that $X=\mathbb{R}^2$ and $U=\mathbb{R}$. Define $L=\langle e^{-t}\partial_t+e^{-x}\partial_x+2\partial_u,e^{-t}\partial_t+\partial_u\rangle$. A short calculation shows that $\mathcal{S}_L=\{(x,u,u_x): 1-u_te^{-t}=0,2-u_te^{-t}-u_xe^{-x}=0\}$. But then, $\mathcal{S}_L=\{(x,u,u_x):u_x=e^x,u_t=e^t\}$. Thus $\mathcal{S}_{L'}=\mathcal{S}_L$ for $L'=\langle \partial_t+e^t\partial_u,\partial_x+e^x\partial_u\rangle$, which is a rectified PDE Lie algebra associated with $L$. Then, $L$ is a rectifiable linear space of vector fields.	
\end{ex}

The following proposition determines straightforwardly, in terms of the distribution $\mathcal{D}^L$, when $L$ is a rectifiable linear space of vector fields.

\begin{prop} A linear space of vector fields $L$ is rectifiable if and only if the distribution $\mathcal{D}^L$ is involutive of rank $p$ and its projection, via $\pi^0_{-1}$, is $TX$.
        \end{prop}
\begin{proof}  Assume that $L$ is rectifiable. By Definition \ref{Def:RLS} and Proposition \ref{Prop:Coi}, one has that $L$ and the rectified PDE Lie algebra $L'$ span the same distribution $\mathcal{D}$.  In view of Proposition \ref{LemBas}, the distribution $\mathcal{D}$ is $p$-dimensional and projects onto $TX$ via $\pi^0_{-1}$. 
	
	   Let us prove the inverse. If $\mathcal{D}^L$ projects onto a distribution $TX$, then there must exist $p$ vectors on $T_{(x,u)}(X\times U)$, for arbitrary $(x,u)\in X\times U$, projecting onto $\partial/\partial x^k$ for $k=1,\ldots,p$, respectively. Therefore, $\mathcal{D}^L$ admits $p$ vector fields of the form
        $$
        Z_k=\frac{\partial}{\partial x^k}+\sum_{\alpha=1}^q\varphi^\alpha_k(x,u)\frac{\partial}{\partial u^\alpha},\qquad k=1,\ldots,p.
        $$
        Since $\mathcal{D}^L$ has rank $p$, the elements of  $L'=\langle Z_1,\ldots, Z_p\rangle$ generate the distribution $\mathcal{D}^L$ and, since $\mathcal{D}^L$ is involutive, one has in view of Lemma \ref{Lem:InvAbe} that $L'$ is $p$-dimensional and Abelian. Consequently, $L$ is rectifiable. 
        \end{proof}

The following corollary is an immediate consequence of previous results that will be useful in the remaining sections of our work.
\begin{cor}\label{NorRel} If $L$ is a rectifiable family of vector fields, then ${\mathcal{S}^n_L}$ is an integral section of $\pi^n_0$, the prolongations of the elements of $L$ to $J^n$ span the tangent space to $\mathcal{S}^n_L$, and $\mathcal{S}^n_L$ is a locally solvable differential equation.
 \end{cor}

\section{Conditional symmetries}\label{Sec:CS}

This section introduces the theory of conditional symmetries with special emphasis on the Clairin approach. Next,
we use the properties of the submanifolds ${\mathcal{S}^n_L}$ obtained in previous sections to investigate the Lie algebras of conditional symmetries of a system of PDEs $\mathcal{S}_\Delta$. Recall that our aim is to restrict the analysis of the Lie algebras of Lie symmetries of  $\mathcal{S}_\Delta$ to determining the Lie algebras of symmetries of a system of PDEs described by ${\mathcal{S}^n_L}\cap \mathcal{S}_\Delta$ where $\mathcal{S}^n_L$ must be an integrable section of $\pi^n_0$, namely a locally solvable $n$-th order system of PDEs. 

In general, ${\mathcal{S}^n_L}\cap \mathcal{S}_\Delta$ does not need to be a manifold and it can even be empty, e.g. the problem on $X=\mathbb{R}^2$, $U=\mathbb{R}$ given by 
$$
\Delta^1=u_{x},\quad \Delta^2=u_{t},\qquad L=\langle \partial_x+\partial_u,\partial_t+u
\partial_u\rangle,
$$
gives rise to an empty set $\mathcal{S}^1_L\cap \mathcal{S}_\Delta$. Hence, the corresponding system of PDEs has no solutions. 
This problem is not commented in applications, which are indeed focused on those cases where $\mathcal{S}^n_L\cap \mathcal{S}_\Delta$ is a manifold and admits particular solutions of physical relevance \cite{CD13,Cl03,GR95,GMR97}. 

We hereafter assume that $\mathcal{S}^n_L$ and $\mathcal{S}_\Delta$ are {\it non-vacuously transversal}, i.e. $\mathcal{S}^n_L\cap \mathcal{S}_\Delta$ is not empty and at every point ${\bf j}^n_x\sigma\in \mathcal{S}^n_L\cap \mathcal{S}_\Delta$ one has that  $T_{{\bf j}^n_x\sigma}\mathcal{S}^n_L+T_{{\bf j}^n_x\sigma}\mathcal{S}_\Delta=T_{{\bf j}^n_x\sigma}J^n$. This turns $\mathcal{S}^n_L\cap \mathcal{S}_\Delta$ into a submanifold of $J^n$ \cite{AMR88}. As illustrated by examples in this work and other previous ones \cite{GMR97,GR95}, this assumption seems to be reasonable. In particular, if $\Delta:J^n\rightarrow \mathbb{R}^s$, then $\dim \mathcal{S}_\Delta=\dim J^n-s$ and $\dim \mathcal{S}^n_L=q$ with $s\leq q$ for the cases of interest in the Clairin theory of conditional symmetries.

The following notion of a Lie algebra of conditional symmetries represents an intrinsic formulation of the definition given in  \cite[Definition 2.2]{GR95}. 

\begin{defn}\label{CLS} A {\it Lie algebra of conditional symmetries} of an $n$-th order system of PDEs  $\mathcal{S}_\Delta\subset J^n$ is a Lie algebra $L$ of vector fields on $X\times U$ such that ${\bf j}^nZ$ is tangent to $\mathcal{S}_\Delta\cap \mathcal{S}^n_L$ for every $Z\in L$.
\end{defn}

As we did not assume that $\mathcal{D}^L$ projects onto $TX$ or its rank is $p$, the previous definition does not need to give rise to a section $\mathcal{S}^n_L$. As a consequence, this definition is more general than that of the Clairin approach and it covers more general Lie algebras of conditional symmetries \cite{GMR97,GR95,Zi93}. Relevantly, our definition is purely geometrical and it does not rely on coordinates expressions of characteristics or the basis of $L$ \cite{Zi93}.

Let us comment on the definition of a Lie algebra of conditional symmetries and other related notions.

We say that a vector field $Y$ is a {\it conditional Lie symmetry} of the system of PDEs given by $\mathcal{S}_\Delta$ if $L=\langle Y\rangle $ is a Lie algebra of conditional symmetries of $\mathcal{S}_\Delta$. This amounts to saying that $Y$ is a {\it nonclassical infinitesimal symmetry} in the terminology employed in \cite{OV95}.

Note that if $Y=\sum_{i=1}^p\xi^i\partial/\partial x^i+\sum_{\alpha=1}^q\varphi^\alpha\partial/\partial u^\alpha$ on $X\times N$, %the function $\sum_{i=1}^p(\xi^i)^2$ is non-vanishing, and $Y$ is a Lie symmetry of $\mathcal{S}_\Delta\subset J^n$, 
then
$$
({\bf j}^n\,Y)\Delta^\mu=0,\qquad \mu=1,\ldots,s,
$$
and %, due to the novanishing of $(\sum_{i=1}^p\xi^i)^2$, Theorems \ref{Th:SecSL1} and \ref{Th:SecSL} show that $\mathcal{S}^n_L$ is a section of $\pi^n_0$ and
Corollary \ref{LTan} ensures that ${\bf j}^n\,Y$ is also  tangent to $\mathcal{S}^n_L$. Consequently, ${\bf j}^n\,Y$ is tangent to $\mathcal{S}^n_L\cap \mathcal{S}_\Delta$ and $Y$ becomes a conditional Lie symmetry of $\mathcal{S}_\Delta$. %Since the given condition on the coefficients $\xi^i$ of $Y$ is quite mild,  
Therefore, Lie point symmetries of $\mathcal{S}_\Delta$ induce a conditional Lie symmetry. Nevertheless, a conditional Lie symmetry does not necessarily give rise to a Lie point symmetry. In particular, there exist systems of PDEs that do not admit any Lie point symmetry while admitting conditional Lie symmetries (cf. \cite{MR03} in references therein).

Many Lie algebras of conditional symmetries in the literature \cite{GMR97,GR95} admit a basis of the form (\ref{B}). Nevertheless, we want to study the theory of Lie algebras of conditional symmetries given by   rectifiable PDE Lie algebras of vector fields. To start with, the following theorem shows that a rectifiable PDE Lie algebra of vector fields consists of conditional Lie symmetries if and only if an associated rectified PDE Lie algebra is a Lie algebra of conditional Lie symmetries.

\begin{thm}\label{Th:Symp} A rectifiable PDE Lie algebra of vector fields $L$ is a Lie algebra of conditional symmetries of $\mathcal{S}_\Delta$ if and only if an associated rectified PDE Lie algebra of vector fields $L'$ consists of conditional symmetries.
        \end{thm}
\begin{proof} Let us prove the direct part of the theorem. By definition, if $L$ is a rectifiable PDE Lie algebra of vector fields, then there exists a rectified PDE Lie algebra $L'$ such that $\mathcal{S}^n_L=\mathcal{S}^n_{L'}$. Let $Y_1,\ldots,Y_l$ be a basis of $L$ and let $Z_1,\ldots,Z_p$ be a basis of $L'$. Hence, to prove that $L$ is a Lie algebra of conditional Lie symmetries if and only if an associated rectifiable $L'$ is, it is enough to show  that 
        $$
        ({\bf j}^n\, Z_k)\Delta|_{\mathcal{S}^n_L\cap \mathcal{S}_\Delta}=0,\,\, k=1,\ldots,p\,\,\Longleftrightarrow  \,\, ({\bf j}^n\, Y_k)\Delta|_{\mathcal{S}^n_L\cap \mathcal{S}_\Delta} =0,\,\,
k=1,\ldots,l\,,
$$

                By virtue of Proposition  \ref{Prop:Coi}, since $L'$ is a rectified PDE Lie algebra associated with $L$, both spaces of vector fields span  the same distribution, namely $\mathcal{D}^L=\mathcal{D}^{L'}$ and  $Y_j=\sum_{k=1}^p\xi^k_jZ_k$ for $j=1,\ldots,l$ and some functions $\xi^k_j\in C^\infty(J^0)$ with $k=1,\ldots,p,j=1,\ldots,l$. Lemma \ref{CharacLin} shows that 
 $
{\bf j}^nY_j%=\sum_{k=1}^p\xi^k_j\left(Z_k+\sum_{\alpha=1}^q\sum_{1\leq |J|\leq n-1}(-u^\alpha_{J,k})\frac{\partial}{\partial u^\alpha_J}\right)
=\sum_{k=1}^p\xi^k_j{\bf j}^nZ_k
$   on $\mathcal{S}_L^n$. Hence, the vector fields ${\bf j}^nY_j$ vanish on $\Delta^\mu$ over $\mathcal{S}^n_L\cap \mathcal{S}_\Delta$  if the $Z_1,\ldots,Z_p$ do. The converse is analogous.             
        \end{proof}

Consequently, the obtention of  a rectifiable PDE Lie algebra of conditional symmetries reduces to the study of rectified PDE Lie algebras of conditional symmetries. Conversely, the knowledge of a rectified PDE Lie algebra allows us to generate many other rectifiable families of vector fields.  This fact has been overlooked  in the literature \cite{GMR97,GR95,Ol92} and it will be demonstrated in Section \ref{Appl}, where we apply our techniques to physically motivated examples. In fact, this is a generalisation of the cases studied in the above-mentioned works. 

Since previous comments relate rectifiable PDE Lie algebras of conditional symmetries to the standard Clairin theory of conditional symmetries dealing with rectified PDE Lie algebras of conditional symmetries \cite{Cl03,GMR97,GR95}, the elements of such rectifiable PDE Lie algebras can be called {\it Clairin conditional symmetries}.

Let us now study the determination of rectified PDE Lie algebras of conditional symmetries. In this case, the system of PDEs associated with ${\mathcal{S}^n_L}$ is given, in an appropriate coordinate system where a basis of $L$ takes the form (\ref{B}), by
\begin{equation}\label{ConSymSys}
\frac{\partial u^\alpha}{\partial x^i}=\phi^\alpha_i(x,u),\qquad i=1,\ldots,p,\quad \alpha=1,\ldots,q,
\end{equation}
and all the total derivatives of this system with respect to the $D_K$ with $|K|\leq n-1$ can be determined. Recall that in view of Proposition \ref{ZCC-LS} the compatibility condition for (\ref{ConSymSys}), which in coordinates takes the form (\ref{DiffCon}), amounts to the fact that $L$ is Abelian. From a practical point of view, this condition and the fact that $\mathcal{S}^n_L$ is locally solvable (due to Corollary \ref{NorRel}) ensure that every solution of (\ref{ConSymSys}) gives rise to a solution of $\mathcal{S}^n_L$ and vice versa. For this reason, the space of solutions of (\ref{ConSymSys}) and the $n$-th order system of PDEs $\mathcal{S}^n_L$ is the same and one can restrict oneself to (\ref{ConSymSys}) in applications. This is the approach used in works on conditional symmetries \cite{Cl03,GMR97,GR95}. The solutions of the system $\mathcal{S}_\Delta\cap {\mathcal{S}^n_L}$ are therefore given by the following system of PDEs:
\begin{equation}\label{ConComplete}
\Delta^\mu(x^i,u^\alpha, u^\alpha_K)=0,\quad \frac{\partial u^\alpha}{\partial x^i}=\phi^\alpha_i(x,u),
\end{equation}
with $i=1,\ldots,p,\alpha=1,\ldots,q,\mu=1,\ldots,s$.

If $L$ is a Lie algebra of conditional symmetries of $\mathcal{S}_\Delta$, then one has the additional condition
$$
{\bf j}^nZ\Delta^\mu|_{\mathcal{S}_\Delta\cap \mathcal{S}_L}=0,\qquad \mu=1,\ldots,s,\qquad \forall Z\in L.
$$
The descriptions of these conditions in coordinates are frequently called the {\it invariant surface conditions} (see \cite{An96}).

Consequently, the obtention of a rectified PDE Lie algebra of conditional symmetries requires the solving of a nonlinear system of PDEs, which is in general much more complicated than the standard system of linear PDEs employed to obtain classical Lie point symmetries \cite{St89}. Nevertheless, recall that this method is applied when standard Lie point symmetries do not provide enough information about the system of PDEs under study.
Moreover, the conditional symmetry approach may offer Lie symmetries of particular families of solutions of the original system of PDEs, which may not be related to Lie symmetries of the general solutions \cite{GMR97}, and the DCs can be used to obtain particular solutions of the initial system of PDEs, which is a topic that was studied in Section \ref{ZCC-LS}.

The following theorem allows us to simplify the system of PDEs necessary to obtain Lie algebras of conditional symmetries of $\mathcal{S}_\Delta$ in the Clairin approach.

\begin{thm}\label{Sym} Let $L$ be a rectifiable PDE Lie algebra of vector fields and let $L'$ be an associated rectified PDE Lie algebra with a basis
$
Z_k=\partial_k+\sum_{\alpha=1}^q\phi^\alpha_k\partial_\alpha
$,
where $k=1,\ldots,p$. Consider the $n$-th order system of PDEs, $\mathcal{S}_\Delta$, described by the zeroes of the functions $\Delta^\mu\in C^\infty(J^n)$, with $\mu=1,\ldots,s$. If $\Delta^\mu(x^j,u^\alpha,D_K\phi^\alpha_j)=0$, with $\mu=1,\ldots,s$ and $|K|\leq n-1$, then $L$ is a Lie algebra of conditional Lie symmetries of $\mathcal{S}_\Delta$.
\end{thm}
\begin{proof} Theorem \ref{Th:Symp} shows that  $L$ is a Lie algebra of conditional symmetries of $\mathcal{S}_\Delta$ if and only if  $L'=\langle Z_1,\ldots,Z_p\rangle $ is a rectified PDE Lie algebra of conditional symmetries of $\mathcal{S}_\Delta$. 

Let us prove that $L'$ is a Lie algebra of conditional symmetries if it satisfies the conditions of our theorem. Take a point ${\bf j}^n_x\sigma\in \mathcal{S}^n_{L}\cap \mathcal{S}_\Delta$. Since $L'$ is a rectified PDE Lie algebra, Corollary \ref{NorRel} states that  $\mathcal{S}^n_L$ is a locally solvable $n$-th order system of PDEs, every ${\bf j}^nZ$, with $Z\in L'$, is tangent to $\mathcal{S}^n_{L'}$, and the ${\bf j}^nZ$ is tangent to $\mathcal{S}^n_{L'}$. Since ${\bf j}^nZ$ is the prolongation to $J^n$ of a vector field on $X\times U$, one has that the one-parameter group of diffeomorphisms $\phi^{(n)}_t$ of ${\bf j}^nZ$ maps holonomic sections on $J^n$ into holonomic sections. The local solvability of $\mathcal{S}^n_L$ ensures that there exists a solution $u(x)$ of $\mathcal{S}^n_{L'}$ such that ${\bf j}^n\sigma$, with $\sigma(x)=(x,u(x))$, passes through ${\bf j}^n_x\sigma$.  Since ${\bf j}^nZ$, for $Z\in L'$ is tangent to $\mathcal{S}^n_{L'}$,  one has that $\phi^{(n)}_t{\bf j}^n\sigma={\bf j}^n\sigma_t$, where $\sigma_t(x)=(x,u_t(x))$ and the $u_t(x)$ are solutions of $\mathcal{S}^n_{L'}$. Hence, we can write
$$
{\bf j}^n\sigma_t=(x^j,u_t^\alpha(x),D_K\phi^\alpha_i(x,u_t(x))),\qquad |K|\leq n-1.
$$
By the assumptions we have made, $\Delta^\mu({\bf }{\bf j}^n\sigma_t)=0$ for $\mu=1,\ldots,s$. Thus,
$$
0=\frac{d}{dt}\bigg|_{t=0}\Delta^\mu({\bf }{\bf j}^n\sigma_t)={\bf j}^nZ\Delta^\mu({\bf j}^n_x\sigma)=0,\qquad \mu=1,\ldots,s.
$$
Consequently, ${\bf j}^nZ$ is tangent to the intersection $\mathcal{S}_{L'}^n\cap \mathcal{S}_\Delta$ and $L$ becomes a Lie algebra of conditional symmetries.
\end{proof}
\section{Conditional symmetries and PDE Lie systems}\label{CSPDE}

The study of conditional symmetries requires the solving of a nonlinear system of PDEs (\ref{ConComplete}), which is more complicated than the  linear one appearing in the determination of classical Lie point symmetries \cite{St89}. As mentioned previously, the solution of (\ref{ConComplete}) is still justified by the fact that we assume that Lie point symmetries are not sufficient to study the properties of the system of PDEs, $\mathcal{S}_\Delta$, under study. 

In this work, we improve the approach initiated in \cite{GR95} and suggest appropriate assumptions allowing for the description of conditional Lie symmetries through the so-called {\it PDE Lie systems}, which provides many techniques to obtain the solutions of such nonlinear PDEs, such as the reduction and integration methods (cf. \cite{CGL18,Dissertationes,Ra11}). 

Consider a family of functions on the coordinates of $u$ that is closed relative to linear combinations, products and derivatives in terms of the coordinates of $u$. For instance, consider the families 
\begin{itemize}
	\item$ \mathcal{A}_1=\{\text{Polynomials in the variables } u^\alpha\}.
	$
	\item $ \mathcal{A}_2=\{\text{Polynomials in the variables } e^{u^\alpha}\}.
	$
	\item $ \mathcal{A}_3=\{\text{Linear combinations  of the functions } 1,\cos(nu^\alpha),\sin(nu^\alpha),n\in \mathbb{N}\}.
$
	\end{itemize}

Let us consider the family $\mathcal{A}_1$. Our methods can be accomplished similarly for $\mathcal{A}_2,\mathcal{A}_3$ or other sets of families satisfying the described properties.
Assume now that  the functions $\phi^\alpha_j$, which depend on the dependent and independent coordinates,  admit a polynomial expansion in the dependent coordinates of $u$ with coefficients given by functions depending on the independent variables, namely
\begin{equation}\label{Assump}
\phi_j^\alpha=\sum_{\widetilde{K}}a^\alpha_{j\widetilde{K}}(x)u^{\widetilde{K}},\qquad j=1,\ldots,p\quad \alpha=1,\ldots,q.
\end{equation}
where $\widetilde{K}=(\widetilde{k}_1,\ldots,\widetilde{k}_q)$ with $\widetilde{k}_i\in \mathbb{N}\cup \{0\}$.  We define $u^{\widetilde{K}}=(u^1)^{\tilde{k}_1}\cdot\ldots\cdot (u^q)^{\tilde{k}_q}$, where the $a^\alpha_{j\widetilde{K}}(x)$ are
certain $x$-dependent functions, and the sum in (\ref{Assump}) is over arbitrary  multi-indices $\widetilde{K}$. In this case, the system determining conditional symmetries (\ref{ConSymSys}) reads 
\begin{equation}\label{ConSymPDELie}
\frac{\partial u^\alpha}{\partial x^j}=\sum_{\widetilde{K}}a^\alpha_{j\widetilde{K}}(x)u^{\widetilde{K}},\qquad j=1,\ldots,p,
\quad \alpha=1,\ldots,q.
\end{equation}

Instead of assuming that $\phi^\alpha_j$ is a second-order polynomial as in \cite{GR95} or making assumptions on the reductions of this system of differential equations as in \cite{GMR97}, we propose to consider that equation (\ref{ConSymPDELie})  gives rise to a general polynomial PDE Lie system \cite{GKO92}, which can be written in the form of an integrable system (in the sense of satisfying the compatibility condition) given by
\begin{equation}\label{PDELiesystem}
\frac{\partial u^\alpha}{\partial x^j}=\sum_{\beta=1}^rb^{\alpha\beta}_{j}(x)X_\beta,\qquad j=1,\ldots,p,\quad \alpha=1,\ldots,q,
\end{equation}
for certain $x$-dependent functions $b^{\alpha\beta}_{j}(x)$ and  vector fields $X_1,\ldots, X_r$ spanning a finite-dimensional Lie algebra of vector fields.  The Lie algebra spanned by $X_1,\ldots, X_r$ is called a {\it Vessiot--Guldberg Lie algebra} of the PDE Lie system \cite{CGL18,Dissertationes}. The Vessiot--Guldberg Lie algebra shows whether the PDE Lie system can be straightforwardly integrated \cite{Ra11} or allows for the construction of {\it superposition rules}, which enables ones to obtain the general solution of the PDE Lie system from a finite set of particular solutions and a set of constants \cite{Dissertationes,Ra11}.

If we restrict the expansion of $\phi^\alpha_j$ to the case where its multi-indices satisfy $|K|\leq 2$, then the right-hand side of (\ref{ConSymPDELie}) becomes a second-order polynomial, which suggests that we impose conditions on the $a^\alpha_{j\widetilde{K}}(x)$ to match the form of the so-called {\it partial differential matrix Riccati  equations} \cite{CGL18,CGM07,FGRSZ99}, which is an integrable first-order system of PDEs (in the sense of satisfying the compatibility condition) of the form
\begin{equation}\label{PRE}
\frac{\partial u}{\partial x_j}=A_j(x)+B_j(x)u+uC_j(x)+uD_j(x)u,\qquad j=1,\ldots,p,
\end{equation}
where $u$ is a vertical vector $(u^1,\ldots, u^q)\in \mathbb{R}^q$, the $A_j(x)$ belong to $ \mathbb{R}^q$, and $B_j(x),C_j(x),D_j(x)$ are $q\times q$ matrices for every $j=1,\ldots,p$. In this case, the Vessiot--Guldberg Lie algebra, $V_{\rm MR2}$, is isomorphic to $\mathfrak{sl}(q+1,\mathbb{R})$ (cf. \cite{GL17,PW83}).

Partial differential matrix Riccati equations appear, for instance, in the study of the Wess-Zumino-Novikov-Witten
(WZNW) equations \cite{CGL18}. They also appear in the study of B\"acklund transformations of relevant systems of PDEs \cite{GR95}. 

The previous conditions on the coefficients $a^\alpha_{j\widetilde{K}}$ are quite general. If $q=1$ all PDE Lie systems are such that a change of variables may map the Lie algebra spanned by the $X_\alpha$ onto $\langle \partial/\partial u,u\partial /\partial u,u^2\partial/\partial u\rangle$ (see \cite{Li80}). When $q=2$, not all PDE Lie systems can be mapped by means of a change of variables on $U$ onto a partial differential matrix Riccati equation (cf. \cite[Table 4]{GL17}). For instance, no change of variables on $U$ can map a generic PDE Lie system with a Vessiot--Guldberg Lie algebra of vector fields $V$ into a matrix Riccati equation when $\dim V>\dim V_{\rm MR2}$.
Such Lie algebras $V$ appear, for instance, in the classification of Lie algebras of vector fields in the plane \cite{GKO92,PW83}. In particular,  one can assume that the $\phi^\alpha_j$ are polynomials of any order (cf. \cite{GKO92} or \cite[Table 4]{GL17}). For instance, one may consider the PDE Lie system
$$
{\small
\begin{aligned}
\frac{\partial u^1}{\partial x^j}&=b^1_i(x)+c^1_i(x)u^1+e_i^1(x)(u^1)^2,\\
\frac{\partial u^2}{\partial x^j}&=b^2_i(x)+c^2_i(x)u^1+ d^2_i(x)u^2+e_i^1(x)r u^1u^2+f^{(1)}_i(x)(u^1)^2+\ldots+f^{(3)}_i(x)(u^1)^3,
\end{aligned}}
$$
with $j=1,\ldots,p$ and where the $x$-dependent functions are chosen so that the system satisfies the compatibility conditions. This PDE Lie system is related to a Vessiot--Guldberg Lie algebra of type ${\rm I}_{20}$ (cf. \cite{GL17}). Consequently, there exists no diffeomorphism on $\mathbb{R}^2$ on the plane mapping this PDE Lie system into a subcase of (\ref{PRE}) because, for instance, the dimension of ${\rm I}_{20}$ is larger than the dimension of $V_{\rm MR2}$. Obviously, studying PDE Lie systems more general than partial differential matrix Riccati equations offers new possibilities for the study of conditional symmetries of PDEs.

Once the form of the system (\ref{ConSymPDELie}) has been established, it is convenient to look for methods to determine Lie algebras of conditional symmetries. In particular, we first focus on the case when the Lie algebra of conditional symmetries is spanned by
$$
Z_k=\frac{\partial}{\partial x_k}+\sum_{\alpha=1}^q\phi_k^\alpha\frac{\partial}{\partial u^\alpha},\qquad k=1,\ldots,p.
$$
As already mentioned, this amounts to the fact that the Lie algebra must be Abelian. Recall also that this latter fact ensures that the differential constraints (\ref{ConSymSys}) give rise to an integrable first-order system of PDEs that therefore admits solutions for every initial condition $(x_0,u_0)$. In coordinates, this condition amounts to the system of PDEs (\ref{DiffCon}) and the substitution of the expansion (\ref{Assump}) into the $n$-th order system of PDEs $\Delta^\mu(x,u,u_K)=0$ leads to a series of conditions
$$
\Delta^\mu(x^j,u^\alpha,D_K\phi^\alpha_j)=0,\qquad \mu=1,\ldots,s.
$$
More specifically, the substitution of (\ref{Assump}) into the compatibility condition (\ref{DiffCon}) of the system (\ref{ConSymPDELie}) gives
\begin{multline}\label{Comp}
a^\alpha_{j\widetilde{K},k}u^{\widetilde{K}}+\sum_{\beta=1}^q\sum_{\widetilde{K}}a^\alpha_{j\widetilde{K}}u^{\widetilde{K}-\i_\beta}\tilde{k}_\beta\sum_{\widetilde{R}}a^\beta_{k\widetilde{R}}u^{\widetilde{R}}\\
-a^\alpha_{k\widetilde{K},j}u^{\widetilde{K}}-\sum_{\beta=1}^q\sum_{\widetilde{K}}a^\alpha_{k\widetilde{K}}u^{\widetilde{K}-\i_\beta}\tilde{k}_\beta\sum_{\widetilde{R}}a^\beta_{j\widetilde{R}}u^{\widetilde{R}}=0,
\end{multline} 
where $\widetilde{R}$ is an arbitrary multi-index for the polynomials on $u$.
To ensure that the above expressions are zero for every $u$, we must impose that the coefficients of the expansion in different polynomials in the variables $u^\alpha$ of (\ref{Assump}) vanish. This in turn can be considered as a first-order system of PDEs on the coefficients. This can be summed up by writing that 
$$
E^r(x,a^\alpha_{j\widetilde{K}},a^\alpha_{j\widetilde{K},k})=0
$$
for a certain set of $r$ functions $E^r$.
%\footnote{In the work \cite{GR95} there is a typo in the description of these conditions, which need not be necessarily given only by $r\leq p$ differential equations.}.
 Additionally, one has to recall that the $a^\alpha_{j\widetilde{K}}$ must still satisfy 
\begin{equation}\label{NL}
\Delta^\mu\left(x^j,u^\alpha,D_K\left(\sum_{\widetilde{K}}a^\alpha_{j\widetilde{K}}(x)u^{\widetilde{K}}\right)\right)=0,\qquad \mu=1,\ldots,s.
\end{equation}
At this point, we assume that the functions $\Delta^\mu$ are polynomials with $x$-dependent coefficients in the functions  of $\mathcal{A}_1$ and their total derivatives of arbitrary order. In view of our assumptions for $\mathcal{A}_1$, this implies that $(\ref{NL})$ can be written as a polynomial  of the form
$$
\sum_{\widetilde{R}}f_{\widetilde{R}}(x,D_Ka^\alpha_{j\widetilde{K}})u^{\widetilde{R}}=0.
$$
By assuming that the coefficients $a^\alpha_{j\widetilde{K}}$ are solutions of the family of differential equations $f_{\widetilde{R}}(x,D_Ka^\alpha_{j\widetilde{K}})=0$, for all the multi-indices $\widetilde{R}$, and (\ref{Comp}), one obtains that all solutions of the PDE Lie system (\ref{PDELiesystem}) become solutions of the higher-order system of PDEs under analysis, namely $\Delta=0$. Hence, all solutions of our PDE Lie system (\ref{PDELiesystem}) become solutions of the initial system of PDEs under investigation. Moreover, the form of the PDE Lie system gives rise to Lie symmetries of the PDE Lie system, which implies that they can be employed to solve it.

\section{Applications}\label{Appl}
This section is aimed at illustrating the results given in previous parts of the work through three different physical models: a nonlinear wave equation, a Gauss-Codazzi equation for minimal surfaces, and a generalised Liouville equation.

\subsection{Nonlinear wave equation}
Let us study the conditional symmetries of a nonlinear wave equation \cite{GR95,Sh78}
\begin{equation}\label{SG}
u_{x_1{x_2}}=b(u),
\end{equation}
where $b(u)$ is, so far, an underdetermined function depending on $u$.
Our study will specify several theoretical and practical details not explained in \cite{GR95}.
The nonlinear wave-equation (\ref{SG}) is a PDE on the second-order jet bundle $J^2$ relative to $X=\mathbb{R}^2$ and $U=\mathbb{R}$. 

Consider a rectified PDE Lie algebra $L$ of conditional symmetries generated by the vector fields
$$
Z_1=\frac{\partial}{\partial x_1}+\psi(x_1,{x_2},u)\frac{\partial}{\partial u},
\qquad Z_2=\frac{\partial}{\partial {x_2}}+\varphi(x_1,{x_2},u)\frac{\partial}{\partial u},
$$
where $\varphi$ and $\psi$ are, for the time being, some functions on $X\times U$ to be determined. In any case, $\mathcal{D}^L$ is a projectable regular distribution of rank two. Therefore, the characteristic system in $J^1$ is given by the conditions
\begin{equation}\label{SG1}
 u_{x_1}=\psi(x_1,{x_2},u),\qquad u_{x_2}=\varphi(x_1,{x_2},u)
\end{equation}
and their total derivatives
\begin{equation}\label{SGext}
\begin{gathered}
u_{{x_2}x_1}=D_{x_1}\varphi({x_1},{x_2},u),\\ u_{{x_1}x_2}=D_{x_2}\psi({x_1},{x_2},u),\,\, u_{{x_2}{x_2}}=D_{x_2}\varphi({x_1},{x_2},u),\,\, u_{{x_1}{x_1}}=D_{x_1}\psi({x_1},{x_2},u),
\end{gathered}
\end{equation}
give rise to the characteristic system $\mathcal{S}_L^2\subset J^2$. In view of Proposition \ref{ZCC-LS}, the condition $[Z_1,Z_2]=0$ amounts to the compatibility condition for (\ref{SG1}), which in turn ensures that $u_{{x_2}{x_1}}=D_{x_1}\varphi({x_1},{x_2},u)=D_{x_2}\psi({x_1},{x_2},u)=u_{{x_1}{x_2}}$, which yields that $\mathcal{S}^2_L\subset J^2$ is a section of $\pi^2_0$ by virtue of Theorem \ref{Th:SecSL}. 

In the Clairin approach to Lie algebras of conditional symmetries, it is assumed that $\varphi$ and $\psi$ are such that the first-order system of PDEs (\ref{SG1}) is integrable. In view of Proposition \ref{ZCC-LS} and Theorem \ref{Th:SecSL}, system $\mathcal{S}^2_L$ is locally solvable. In view of Corollary \ref{Simplification}, the solutions to the system (\ref{SG}), (\ref{SG1}), (\ref{SGext}) are the same as the solutions to (\ref{SG}) and (\ref{SG1}). This is why works dealing with applications of conditional symmetries in the Clairin approach restrict themselves to studying $\mathcal{S}^1_L$ instead of $\mathcal{S}^n_L$ (cf. \cite{GMR97,GR95}).

In order to use PDE Lie systems to study the conditional symmetries of (\ref{SG}), consider the expansion of $\psi$ and $\varphi$ up to second order in the dependent variable $u$. Then,
\begin{equation}
\label{expansion}
\begin{gathered}\psi=a_2({x_1},{x_2})u^2+a_1({x_1},{x_2})u+a_0({x_1},{x_2}),\\
\varphi=b_2({x_1},{x_2})u^2+b_1({x_1},{x_2})u+b_0({x_1},{x_2}), 
\end{gathered}\end{equation}
for functions $a_2({x_1},{x_2}),a_1({x_1},{x_2}),a_0({x_1},{x_2}),b_2({x_1},{x_2}),b_1({x_1},{x_2}),b_0({x_1},{x_2})$ whose particular form must be determined to ensure that $L$ is a Lie algebra of conditional symmetries

In view of (\ref{SG}) and (\ref{SG1}), we get that $D_{x_2}\psi=D_{x_1}\varphi=b(u)$. Therefore, one can restrict oneself to the particular case $b(u)=c_0+c_1u+c_2u^2+c_3u^3$ for certain constants $c_0,c_1,c_2,c_3\in \mathbb{R}$ (other more general cases could also be considered). Note that therefore the system of PDEs (\ref{SG}) satisfies the conditions of the formalism given in Section \ref{CSPDE} for $\mathcal{A}_1$ and we can try to obtain solutions to this equation by means of solutions of a PDE Lie system of the form (\ref{SG1}).

By substituting the expansion (\ref{expansion}) and the differential constraints (\ref{SG1}) into the wave-equation (\ref{SG}), one obtains 
\begin{multline*}
2a_2b_2u^3+ (2b_2a_1+b_1a_2+\partial_{x_1}b_2)u^2+(2b_2a_0+b_1a_1+\partial_{x_1}b_1)u\\+a_0b_1+\partial_{x_1}b_0=c_3u^3+c_2u^2+c_1u+c_0.
\end{multline*}
To ensure that all solutions of (\ref{SG1}) are solutions of (\ref{SG}), we assume that the coefficients in $u$ are all zero. 

Then, the compatibility condition for (\ref{SG1}) amounts to the fact that
$$
\begin{gathered}
2b_2a_1+b_1a_2+\partial_{x_1}b_2=2a_2b_1+a_1b_2+\partial_{x_2}a_2,\\ 2b_2a_0+b_1a_1+\partial_{x_1}b_1=2a_2b_0+a_1b_1+\partial_{x_2}a_1,\\ a_0b_1+\partial_{x_1}b_0=a_1b_0+\partial_{x_2}a_0.
\end{gathered}
$$ 
%or equivalently, by substituting (\ref{expansion}) and (\ref{SG1}) in $u_{x_2x_1}=b(u)$, one gets 
%$$
%2a_2b_2u^3+ (2a_2b_1+a_1b_2+\partial_{x_2}a_2)u^2+(2a_2b_0+a_1b_1+\partial_{x_2}a_1)u+a_1b_0+\partial_{x_2}a_0=c_3u^3+c_2u^2+c_1u+c_0.
%$$
In view of previous equalities, we obtain
\begin{equation}\label{ExCon}
\begin{gathered}
2a_2b_2=c_3,\\ 2b_2a_1+b_1a_2+\partial_{x_1}b_2=c_2,\quad 2b_2a_0+b_1a_1+\partial_{x_1}b_1=c_1,\quad a_0b_1+\partial_{x_1}b_0=c_0,\\
 2a_2b_1+a_1b_2+\partial_{x_2}a_2=c_2,\quad 2a_2b_0+a_1b_1+\partial_{x_2}a_1=c_1,\quad a_1b_0+\partial_{x_2}a_0=c_0.
\end{gathered}
\end{equation}
If the coefficients of the expansion (\ref{expansion}) satisfy the above conditions, then Theorem \ref{Sym} ensures that $Z_1$ and $Z_2$ are conditional Lie symmetries of (\ref{SG}). It is therefore not necessary to consider the extension of the vector fields $Z_1,Z_2$ to $J^2$, to restrict them to $\mathcal{S}^2_L\cap \mathcal{S}_\Delta$, and to verify whether the functions $\varphi$ and $\psi$ give rise to conditional symmetries.

Let us provide a simple case of conditional symmetries for (\ref{SG}). A particular solution to the equations (\ref{ExCon}), e.g. $a_2=b_2=1,a_0=b_0=a_1=b_1=0$, gives rise to
$$
Z_1=\partial_{x_1}+u^2\partial_u,\qquad Z_2=\partial_{x_2}+u^2\partial_u,
$$
which span a Lie algebra of conditional symmetries of 
\begin{equation}\label{case}
u_{{x_1}{x_2}}=2u^3.
\end{equation}
Moreover, all solutions of 
\begin{equation*}
\frac{\partial u}{\partial x_1}=u^2,\qquad\frac{\partial u}{\partial x_2}=u^2,
\end{equation*}
namely $u=-1/(x_1+x_2+\lambda)$ with $\lambda\in \mathbb{R}$, are particular solutions of (\ref{case}).

We can obtain a new Lie algebra of conditional Lie symmetries by obtaining a rectifiable family of vector fields $L'$ whose $\mathcal{S}^n_{L'}$ will match $\mathcal{S}^n_L$, for instance,
$$
Y_1=\partial_{x_1}+u^2\partial_u,\qquad Y_2=e^{{x_2}/u}(\partial_{x_2}+u^2\partial_u).
$$
In fact, one gets that $[Y_1,Y_2]=-Y_2$.  It is worth noting that the new vector fields can be interesting if they leave invariant a geometric structure on $J^0$ whereas the vector fields $Z_1,Z_2$ do not (see \cite{BBHLS15} for examples of this). 

Note that we could have repeated the whole above procedure by using the set of functions $\mathcal{A}^3$. In that case, we would have added differential constraints given by a PDE Lie system with a Vessiot--Guldberg Lie algebra given by 
$$V=\left\langle \cos(u)\frac{\partial}{\partial u},\sin(u)\frac{\partial}{\partial u},\frac{\partial}{\partial u}\right\rangle.\,\,$$
In turn, this gives rise to the study of the sine-Gordon equation $u_{xt}=\frac 12\sin(2u)$, which admits pseudo-spherical surfaces and gives rise to B\"acklund transformations \cite{Ol92,RS00,Vo91}. Another set of functions $\mathcal{A}$ would have given rise to the Vessiot--Guldberg Lie algebras
$$
V=\left\langle e^u\frac{\partial}{\partial u},e^{-u}\frac{\partial}{\partial u},\frac{\partial}{\partial u}\right\rangle,\qquad
V=\left\langle {\rm ch}(u)\frac{\partial}{\partial u},{\rm sh}(u)\frac{\partial}{\partial u},\frac{\partial}{\partial u}\right\rangle
$$
and new results and Lie algebras of conditional symmetries could be derived for the obtained types of Sinh-Gordon equations \cite{RS00}.
\subsection{Gauss-Codazzi equations}
Let us now apply our formalism to the study of a class of  {\it Gauss--Codazzi equations} \cite{Sc41}, which take the form
\begin{equation}\label{GC}
\partial \bar \partial u+\frac 12 H^2 e^u-2|Q|^2e^{-u}=0,\qquad \bar \partial Q=\frac 12 (\partial H )e^u,
\end{equation}
where $\partial=(\partial/\partial x_1-{\rm i}\partial/\partial x_2)/2,\bar \partial=(\partial/\partial x_1+{\rm i}\partial/\partial x_2)/2$, $u=u(z,\bar z)$ and $H=H(z,\bar z)$ are real functions, and $Q=Q(z,\bar z)$ is a complex valued function. We will focus on the first part of the Gauss--Codazzi equations. Although it is a complex differential equation due to the appearance of $\partial$ and $\bar \partial$, it is simple to see that it can be considered as a real differential equation on the bundle $J^2$ relative to $U=\mathbb{R}$ and $X=\mathbb{R}^2$. 

In order to apply the conditional symmetry theory, we consider the differential constraints 
\begin{equation}\label{CS}
\begin{aligned}
\partial u=\eta_0+\eta_1e^{-u/2}+\eta_2e^{u/2},\qquad \bar\partial u=\bar \eta_0+\bar \eta_1e^{-u/2}+\bar \eta_2e^{u/2},
\end{aligned}
\end{equation}
whose compatibility condition amounts to
\begin{equation}
\label{ZCC}
\begin{aligned}
\bar\partial \eta_0-\partial \bar\eta_0-\eta_1\bar \eta_2+\bar \eta_1\eta_2=0,\\\,\,
\bar\partial \eta_1-\partial \bar\eta_1+\frac {\eta_0\bar \eta_1-\bar \eta_0\eta_1}2=0,\\\,\,
\bar\partial \eta_2-\partial \bar\eta_2+\frac {\eta_2\bar \eta_0-\bar \eta_2\eta_0}2=0.
\end{aligned}
\end{equation}
The system of PDEs (\ref{CS}) can easily be rewritten as a PDE Lie system of the form
$$
\begin{gathered}
\frac{\partial u}{\partial x_1}={\rm Re}(\eta_0)+{\rm Re}(\eta_1)e^{-u/2}+{\rm Re}(\eta_2)e^{u/2},\\
\frac{\partial u}{\partial x_2}={\rm Im}(\bar \eta_0)+{\rm Im}(\bar \eta_1)e^{-u/2}+{\rm Im}(\bar \eta_2)e^{u/2},
\end{gathered}
$$
related to a Vessiot--Guldberg Lie algebra $V=\langle \partial/\partial u,e^{-u/2}\partial/\partial u,e^{u/2}\partial/\partial u\rangle$. This allows us to study two-dimensional Lie algebras of conditional symmetries spanned by the vector fields
\begin{equation}\label{BL}
\begin{gathered}
\frac{\partial}{\partial x_1}+({\rm Re}(\eta_0)+{\rm Re}(\eta_1)e^{-u/2}+{\rm Re}(\eta_2)e^{u/2})\frac{\partial}{\partial u},\\ \frac{\partial}{\partial x_2}+({\rm Im}(\bar \eta_0)+{\rm Im}(\bar \eta_1)e^{-u/2}+{\rm Im}(\bar \eta_2)e^{u/2})\frac{\partial}{\partial u}.
\end{gathered}\end{equation}

To verify whether the vector fields related to our differential constraints give rise to conditional Lie symmetries, we have to substitute (\ref{CS}) in (\ref{GC}). Then, we obtain
\begin{multline}\label{Com}
\bar \partial \eta_0+\frac{\eta_2\bar \eta_1-\eta_1\bar \eta_2}2+\left( \bar\partial \eta_1-\frac {\eta_1\bar \eta_0}2\right)e^{-\frac u2}+\\\left(\bar \partial \eta_2+\frac {\bar \eta_0\eta_2}2\right)e^{\frac u2}+ \frac{H^2+\eta_2\bar \eta_2}2{e^u}-\frac{\eta_1\bar \eta_1+4|Q|^2}2e^{-u}=0.
\end{multline}
Let us assume that all coefficients accompanying the different exponentials of the variable $u$ are zero. This ensures that all solutions of (\ref{CS}) give rise to solutions of the Gauss--Codazzi equations. In view of Theorem \ref{Sym}, this also ensures that the vector fields (\ref{BL}) span a Lie algebra of conditional symmetries. Then, we obtain
\begin{eqnarray}
\bar \partial \eta_0+\frac{\eta_2\bar \eta_1-\eta_1\bar \eta_2}2=0,\label{con1}\\
 \bar\partial \eta_1-\frac {\eta_1\bar \eta_0}2=0,\label{con2}\\
 \bar \partial \eta_2+\frac {\bar \eta_0\eta_2}2=0,\label{con3}\\
 H^2+\eta_2\bar \eta_2=0,\label{con4}\\
 \eta_1\bar \eta_1+4|Q|^2=0\label{con5}.
\end{eqnarray}
Since $H$ is a real-valued function, equation (\ref{con4})  implies that $\eta_2=H=0$ and the Gauss--Codazzi equations (\ref{GC}) show that $Q=Q(z)$. It is immediate to verify that the latter fact, along with (\ref{con1})--(\ref{con3}), allows us to ensure that the compatibility conditions (\ref{ZCC}) are satisfied. 

Let us now obtain the conditional symmetries of this system. Since $\eta_2=0$, one has from (\ref{con1}) that $\eta_0=\eta_0(z)$. The equations (\ref{con1}) can easily be solved to obtain that
$$
\eta_1=\alpha(z)e^{\frac 12\int\bar \eta_0 d\bar z}
$$
for an arbitrary function $\alpha=\alpha(z)$. 
Substituting this in (\ref{con5}), we get 
$$
|Q|^2=-\frac{1}{4}\left|\alpha(z)e^{\frac 12\int \eta_0dz}\right|^2,
$$
which is acceptable since $Q=Q(z)$ is a holomorphic function. Consequently, the initial Gauss--Codazzi equation reduces under previous assumptions to
$$
\partial \bar \partial u+\frac 12\left|\alpha(z)e^{1/2\int \eta_0dz}\right|^2e^u=0,
$$
whereas
\begin{equation}\label{PDELieImm}
\begin{gathered}
\frac{\partial u}{\partial x_1}={\rm Re}(\eta_0(z))+{\rm Re}(\alpha(z)e^{\frac 12\int \bar \eta_0d\bar z})e^{-u/2},\\
\frac{\partial u}{\partial x_2}={\rm Im}(\overline {\eta_0(z)})+{\rm Im}(\overline {\alpha(z)}e^{\frac 12\int \eta_0dz})e^{-u/2},
\end{gathered}
\end{equation}
where we recall that $z=x_1+{\rm i}x_2$.
The latter is a PDE Lie system related to a solvable Vessiot--Guldberg Lie algebra $V_{2}=\langle \partial/\partial u, e^{-u/2}\partial/ \partial u\rangle$. Consequently, it is solvable (cf. \cite{Dissertationes}). In fact, a simple method to solve it goes as follows. Let us rewrite (\ref{PDELieImm}) in terms of the variable $w=e^{u/2}$ as
\begin{equation}
\begin{gathered}
\frac{\partial w}{\partial x_1}={\rm Re}(\eta_0(z))w/2+{\rm Re}(\alpha(z)e^{\frac 12\int \bar \eta_0d\bar z})/2,\\
 \frac{\partial w}{\partial x_2}={\rm Im}(\overline {\eta_0(z)})w/2+{\rm Im}(\overline{ \alpha(z)}e^{\frac 12\int \eta_0dz}).
\end{gathered}
\end{equation}
The homogeneous part of the system reads
\begin{equation}
\frac{\partial w_H}{\partial x_1}={\rm Re}(\eta_0(z))w_H/2,\qquad \frac{\partial w_H}{\partial x_2}={\rm Im}(\overline{\eta_0(z)})w_H/2
\end{equation}
and a particular solution reads
$w_H=e^{\frac 12\int({\rm Re}(\eta_0(z))dx_1+{\rm Im}(\eta_0(z)dx_2)}$. Substituting $w=w_Hw_N$ into (\ref{PDELieImm}), we obtain that 
$$
w_N=\int\frac{1}{2w_H}{\rm Re}(\alpha(z)e^{\frac 12\int \bar \eta_0d\bar z})dx_1+\int\frac{1}{2w_H}{\rm Im}(\overline{ \alpha(z)}e^{\frac 12\int \eta_0dz})dx_2
$$
and the final solution for (\ref{PDELieImm}) follows immediately. It is remarkable that $H=0$ implies that the obtained surfaces are minimal \cite{En68,We66}.

\subsection{A generalised Liouville equation}
Finally, let us study the generalised Liouville equation in $\mathbb{R}^{n+1}$ introduced by Santini \cite{Sa96}. Let $(x_0=t,x)$ be a general point in $\mathbb{R}\times \mathbb{R}^{n}$ and let $\nabla$ be the gradient operator in $\mathbb{R}^n$ relative to the Euclidean metric $\langle\cdot|\cdot\rangle$ on $\mathbb{R}^n$. Then, the generalised Liouville equation is given by 
\begin{equation}\label{GLE}
\frac{\partial}{\partial t}\left(\nabla^2 u-\frac 12 \langle \nabla u|\nabla u\rangle\right)=0,
\end{equation}
where $u=u(t,x)$ is a function on $\mathbb{R}^{n+1}$. 

We make use of the set of functions $\mathcal{A}_2$ and we assume that the conditional symmetry is given by
\begin{equation}\label{DC}
\frac{\partial u}{\partial x_j}=A^0_j(t,x)+A^+_j(t,x)e^{u/2}+A^-_j(t,x)e^{-u/2},\qquad j=0,\ldots,n.
\end{equation}
The reason for this ansatz is that the vector fields 
$$
Z_1=\frac{\partial}{\partial u},\qquad Z_2=e^{u/2}\frac{\partial}{\partial u},\qquad Z_3=e^{-u/2}\frac{\partial}{\partial u},
$$
span a Lie algebra of vector fields isomorphic to $\mathfrak{sl}(2,\mathbb{R})$. Thus, if we choose $A^0_j(t,x),A^+_j(t,x),A^-_j(t,x)$, with $j=0,\ldots,n$, so that (\ref{DC}) is integrable, then we get that (\ref{DC}) is a PDE Lie system.
The $A^0_\alpha(t,x),A^+_\alpha(t,x),A^-_\alpha(t,x)$, with $\alpha=1,\ldots,n$, give rise to three $t$-dependent vector fields  ${\bf A}^0,{\bf A}^+,{\bf A^-}$ on $\mathbb{R}^n$, respectively. Using this and substituting the differential constraints (\ref{DC})  into (\ref{GLE}), we obtain the following equation:
{\small 
\begin{multline*}
\frac{\partial}{\partial t}\left[\nabla {\bf A}^0-\frac 12 |{\bf A}^0|^2\!+\!\left(\nabla {\bf A}^+-\frac12 \langle {\bf A}^+|{\bf A}^0\rangle\right)e^{u/2}\!\right.\\+\left.\!\left(\nabla {\bf A}^--\frac 32 \langle {\bf A} ^-|{\bf A}^0\rangle\right)e^{-u/2}\!-\!|{\bf A}^-|^2e^{-u}\right]=0.
\end{multline*}}
Therefore, the particular conditions
\begin{equation}\label{Sus}
\frac{\partial}{\partial t}\left[\nabla{\bf A}^0-\frac 12 |{\bf A}^0|^2\right]=0,\qquad \nabla {\bf A}^+-\frac12 \langle {\bf A}^+|{\bf A}^0\rangle=0,\qquad {\bf A}^-=0,
\end{equation}
ensure that every solution of the DCs (\ref{DC}) gives rise to a particular solution of  (\ref{GLE}). Using that ${\bf A^-}=0$ and assuming that $A^-_0=0$, the compatibility condition for (\ref{DC}) reads
\begin{equation}\label{ComCon}
\frac{\partial A_k^0}{\partial x_j}-\frac{\partial A^0_j}{\partial x_k}=0,\qquad \frac{\partial A_k^+}{\partial x_j}-\frac{\partial A^+_j}{\partial x_k}+\frac 12(A^+_k A^0_j-A^0_k A^+_j)=0,\qquad 0\leq k<j \leq n.
\end{equation}
The first condition implies that there exists a function $u_p\in C^\infty(\mathbb{R}^{n+1})$ such that 
\begin{equation}\label{potential1}
A^0_j=\frac{\partial u_p}{\partial x^j},\qquad  j=0,\ldots,n.
\end{equation}
In particular, ${\bf A^0}=\nabla u_p$ and substituting this in the first equality of (\ref{Sus}), one obtains that 
$$
\nabla^2 u_p-\frac12|\nabla u_p|^2=\lambda(x) 
$$
and $u_p$ becomes a particular solution of (\ref{GLE}). Using (\ref{potential1}) in the second compatibility condition in (\ref{ComCon}), we obtain that
$$
\frac{\partial \left(A^+_k e^{u_p/2}\right)}{\partial x^j}-\frac{\partial \left(A^+_j e^{u_p/2}\right)}{\partial x^k}=0,\qquad 0\leq k<j\leq n.
$$
Therefore,
$$
\,\,\exists w\in C^\infty(\mathbb{R}^{n+1}),\qquad A^+_j=e^{-u_p	/2}\frac{\partial w}{\partial x^j},\qquad 0\leq j\leq n.
$$
In particular, ${\bf A}^+=e^{-u_p/2}\nabla w$. Using ${\bf A}^+$ in the second expression of (\ref{ComCon}), we obtain
\begin{equation}\label{auxiliar}
\nabla^2 w- \langle \nabla u_p|\nabla w\rangle=0.
\end{equation}
Under previous assumptions, the DCs (\ref{DC}) reduce to 
\begin{equation}\label{DC3}
\frac{\partial u}{\partial x_j}=\frac{\partial u_p}{\partial x_j}+\frac{\partial w}{\partial x_j}e^{(u-u_p)/2},\qquad j=0,\ldots,n.
\end{equation}
As this is a PDE Lie system related to the solvable Lie algebra of vector fields $\langle \partial/\partial u,e^{u/2}\partial/\partial u\rangle$, its general solution can be obtained \cite{Ra11}. In fact, it follows from (\ref{DC3}) that
\begin{equation*}\label{}
 \frac{\partial}{\partial x_j}\left(e^{-(u-u_p)/2}-w\right)=0,\qquad j=0,\ldots,n.
\end{equation*}
and
\begin{equation}\label{Sup}
u(t,x)=u_p(t,x)-2\log (w(t,x)+\lambda),\qquad \lambda\in \mathbb{R},
\end{equation}
is the general solution of the system of equations given by (\ref{GLE}) and (\ref{DC3}). This can be viewed as a generalised type of auto-B\"acklund transformation mapping solutions of the generalised Liouville equation satisfying the DC (\ref{DC}) into new solutions of the same equations. %Additionally, this can also be seen as a so-called {\it $t$-dependent superposition rule} for a system of PDEs, which was recently introduced in \cite{CGL18}.

Let us present several particular solutions. If $n=2$ and $u_p=h(t)$ for any $t$-dependent function $h(t)$, then equation (\ref{auxiliar}) reduces to $\nabla^2w=0$ and $w$ becomes any {\it harmonic function} on $\mathbb{R}^2$, e.g. the real and imaginary parts $\mathfrak{re }f(x_1+{\rm i}x_2,t)$ and $\mathfrak{im}f(x_1+{\rm i}x_2,t)$, respectively, of a $t$-dependent holomorphic function $f(x_1+{\rm i}x_2,t)$. In this case, 
$$
u=h(t)-2\ln[\lambda +\mathfrak{re }f(x_1+{\rm i}x_2,t)+\mathfrak{im}f(x_1+{\rm i}x_2,t)].
$$

The above procedure gives a new approach to  \cite{GMR97}, explains many details not given in there, and avoids several of its {\it ad-hoc} constructions, e.g. our approach does not need any ad-hoc transformation. Let us finally explain how to obtain a particular solution proposed in \cite{GMR97} without a detailed explanation.

Every function $u_p(t,x)=f(t)+g(x)$, for arbitrary functions $f(t)$ and $g(x)$, is a particular solution of (\ref{GLE}). 
In particular, if we assume $\omega(x)=\sum_{j=1}^n\omega_j(x_j)$, then
$$
u_p(t,x)=\ln[g(t)]+\sum_{j=1}^n\ln\left[\frac{d\omega_j}{dx_j}(x_j)\right]
$$
is a particular solution of (\ref{GLE}) and $\omega(x)$ is a solution of (\ref{auxiliar}).  According to (\ref{Sup}), we obtain a solution with the freedom of $n+1$ arbitrary functions of one variable
$$
u(t,x)=\ln\left[\frac{g(t)\prod_{j=1}^n\frac{d\omega_j}{dx_j}(x_j)}{\left[g(t)+\sum_{j=1}^n\omega_j(x_j)+\lambda\right]^2}\right],\qquad \lambda\in \mathbb{R}
$$
of the system of PDEs given by (\ref{GLE}) and (\ref{DC}). This reproduces the multi-mode particular solution given in \cite{GMR97} without a detailed explanation.
\section{Conclusions}
This work has provided a geometric approach to the Clairin theory of conditional symmetries for higher-order systems of PDEs. Many of the hypotheses employed in this theory have been analysed and their geometrical meaning has been explained. In certain cases, it was stated that standard assumptions can be relaxed. Then, the role of PDE Lie systems in the description of conditional symmetries has been analysed. 

After the study of the Clairin theory of conditional symmetries, we believe that we should continue the analysis of other conditional symmetries and to consider the case of conditional symmetries given by the characteristics of non-Lie point symmetries. It seems to us that the geometric formalism developed in this work could be generalised to comprise these more general cases. 

\section{Acknowledgements}
A.M. Grundland was partially supported by the research grant ANR-11LABX-0056-LMHLabEX LMH (Fondation Math\'ematique Jacques Hadamard, France) and
an Operating Grant from NSERC, Canada. J. de Lucas and A.M. Grundland acknowledge partial support from HARMONIA 2016/22/M/ST1/00542 of the National Science Center (Poland). This work was partially accomplished during the stays of A.M. Grundland at the University of Warsaw and of J. de Lucas at the Centre de Recherches Math\'ematiques (CRM) of the University of Montr\'eal. The authors would also like to thank both institutions for their hospitality and attention during their respective stays.

%\bibliographystyle{actapoly}
%\bibliography{biblio}

\end{document}